\documentclass[10pt]{amsart}

\usepackage{comment}

\usepackage{amsmath}
  \usepackage{paralist}
  \usepackage{graphics} %% add this and next lines if pictures should be in esp format
  \usepackage{epsfig} %For pictures: screened artwork should be set up with an 85 or 100 line screen
\usepackage{graphicx}  \usepackage{epstopdf}%This is to transfer .eps figure to .pdf figure; please compile your paper using PDFLeTex or PDFTeXify.

 \usepackage[colorlinks=true]{hyperref}
\hypersetup{urlcolor=blue, citecolor=red}

\newtheorem{theorem}{Theorem}[section]
\newtheorem{corollary}[theorem]{Corollary}

\newtheorem*{main*}{Main Theorem}
\newtheorem{lemma}[theorem]{Lemma}
\newtheorem{proposition}[theorem]{Proposition}

\theoremstyle{definition}
\newtheorem{definition}[theorem]{Definition}
\newtheorem{remark}[theorem]{Remark}

\newtheorem*{theorems}{Theorem}

\newtheorem*{TheoremA}{Theorem A}
\newtheorem*{TheoremB}{Theorem B}
\newtheorem*{TheoremC}{Theorem C}
\newtheorem*{TheoremD}{Theorem D}
\newtheorem*{CorollaryA1}{Corollary A.1}
\newtheorem*{CorollaryA2}{Corollary A.2}
\newtheorem*{CorollaryB1}{Corollary B.1}
\newtheorem*{CorollaryC1}{Corollary C.1}
\newtheorem*{CorollaryC2}{Corollary C.2}
\newtheorem*{CorollaryC3}{Corollary C.3}

\newcommand{\ep}{\varepsilon}

\title[Unstable Entropies and  Variational Principle]
      {Unstable Entropies and  Variational Principle for Partially Hyperbolic Diffeomorphisms}

\def\a{\alpha}
\def\b{\beta}
\def\c{\gamma}   
\def\d{\delta}   %\def\C{\Gamma}
\def\l{\lambda}   
 \def\e{\varepsilon}

\def\ae{-\text{a.e.}\ }

\def\B{{\mathcal B}}
\def\P{{\mathcal P}}
\def\Q{{\mathcal Q}}

\def\diam{\mathop{\hbox{{\rm diam}}}}

\def\loc{{\mathop{\hbox{\footnotesize  \rm loc}}}}

\def\top{{\mathop{\hbox{\footnotesize \rm top}}}}

\def\disp{\displaystyle}

\author[]{Huyi Hu, Yongxia Hua and Weisheng Wu}


\address{Department of Mathematics, Michigan State University, East Lansing, MI 48824, USA}
 \email{hhu@math.msu.edu}

\address{Department of mathematics, South university of Science and Technology of China, 518055 Shenzhen, P.R. China}
 \email{huayx@sustc.edu.cn}

\address{Department of Applied Mathematics, College of Science, China Agricultural University, Beijing, 100083, P.R. China}
 \email{wuweisheng@cau.edu.cn}

\begin{document}

\maketitle
\markboth{Unstable entropies and  variational principle}
{Huyi Hu, Yongxia Hua and Weisheng Wu}
\renewcommand{\sectionmark}[1]{}

\begin{abstract}
We study entropies caused by the unstable part of partially hyperbolic systems.
We define unstable metric entropy and unstable topological entropy,
and establish a variational principle for partially hyperbolic
diffeomorphsims, which states that the unstable topological entropy is
the supremum of the unstable metric entropy taken over all invariant
measures.  The unstable metric entropy for an invariant measure
is defined as a conditional entropy along unstable manifolds,
and it turns out to be the same as that given by Ledrappier-Young,
though we do not use increasing partitions.
The unstable topological entropy is defined
equivalently via separated sets, spanning sets and open covers along
a piece of unstable leaf, and it coincides with the unstable volume growth
along unstable foliation.
We also obtain some properties for the unstable metric entropy such as
affineness, upper semi-continuity and a version of
Shannon-McMillan-Breiman theorem.
%\textcolor[rgb]{1.00,0.00,0.00}{(it seems there are many places missing ``the'')}
\end{abstract}

\section{Introduction}
The difference between partially hyperbolic systems and hyperbolic systems
is the presence of the center direction in the former case.
The original motivation of the paper is to study some ergodic properties
of partially hyperbolic systems that \emph{arise from the hyperbolic part}.
Since entropies are the important invariants measuring the complexity
of the systems, they are good objects to start with.

It is generally agreed that entropies are caused by the expansive part of
dynamical systems.
There are some existing notions for such measurements, including the entropies
given by Ledrappier and Young (\cite{LY2}) from the measure theoretic point
of view and the unstable volume growth given by Hua, Saghin,
and Xia~(\cite{HSX}) from the topological point of view.
Another motivation of the paper is trying to put them in
a framework that is similar to the classical entropy theory.
In this paper we redefine the notion of \emph{unstable metric entropy}
$h_\mu^u(f)$, and define the \emph{unstable topological entropy} $h_{\top}^u(f)$,
and prove a \emph{variational principle} for them. Also,
for the unstable metric entropy, we provide a version of
\emph{Shannon-McMillan-Breiman theorem}.

The \emph{unstable metric entropy} for an invariant measure $\mu$
is defined by using $H_\mu(\vee_{i=0}^{n-1} f^{-i}\a|\eta)$,
where $\a$ is a finite measurable partition of the underlying manifold $M$, and
$\eta$ is a measurable partition consisting of local unstable leaves
that can be obtained by refining a finite partition into pieces of
unstable leaves (see details in Definition~\ref{Defuentropy}).
The definition is about the same as the classical metric entropy,
except for the conditional partition $\eta$, which is used
to eliminate the impact from center directions.
The entropy defined in \cite{LY2} can be regarded as that given by
$H_\mu(\xi|f\xi)$, where $\xi$ is an increasing partition, that is,
$\xi \ge f\xi$, subordinate to unstable manifolds.
We show in Theorem~A that it is identical to
the unstable metric entropy we defined.
We also prove that the unstable metric entropy map, as a function
from the set of all invariant measures to nonnegative real numbers,
is affine and upper semi-continuous (Proposition~\ref{Paffine} and
Proposition~\ref{Pusc}).

Generally speaking, a good notion of entropy should satisfy
a type of Shannon-McMillan-Breiman theorem.  In this paper
we provide a version of \emph{Shannon-McMillan-Breiman theorem
for unstable metric entropy} in Theorem~B.

We define \emph{unstable topological entropy} by using the growth rates
of the cardinality of $(n, \e)$ separated sets or spanning sets
of a local unstable leaf at every point $x$ then taking the supremum over
$x\in M$ (see Definition~\ref{Defutopent1}).
It measures the asymptotic rate of orbit divergence along unstable manifolds.
As the classical topological entropy (see Sections 7.1 and 7.2 in \cite{W}), the cardinality of $(n, \e)$
separated sets or spanning sets can be replaced by a subcover
of an open cover of the form $\vee_{i=0}^{n-1} f^{-i}{\mathcal U}$,
where ${\mathcal U}$ is an open cover of $M$ (Definition~\ref{Defutopent2}).
We show in Theorem~C that the unstable topological entropy we defined
coincides with the volume growth given in \cite{HSX}.

As same as the classical case, we can obtain a \emph{variational principle}
for unstable entropies (Theorem~D).  That is, the unstable topological entropy
is the supremum of unstable metric entropy taken over all
invariant probability measures, as well as all ergodic measures.

Ledrappier and Young (\cite{LY2}) introduced a hierarchy
of metric entropies $h_i=h_i(f)$,
each of which corresponds to a different Lyaponov exponent,
and is regarded as the entropy caused by different hierarchy of
unstable manifolds.  If there are $u$ different positive Lyapunov
exponents in the unstable direction of a partially hyperbolic system,
then $h_u(f)$ gives the unstable metric entropy we define in this paper.
The entropy has a simple form $H_\mu(\xi|f\xi)$,
where $\xi$ is an increasing partition subordinate to unstable manifolds.
Because the partition is increasing,
it is convenient to use sometimes.
However, practically it takes some work to construct such partitions,
such as in \cite{PS} and \cite{Yan}.  In our definition,
instead of $\xi$, we relax the increasing condition and use partitions
$\eta$ that can be obtained by refining any finite partitions
of small diameters to unstable leaves, and is much easier
to construct.  Moreover, the size of the elements of $\eta$
can be uniformly bounded from above and below, while the size of elements
of $\xi$ could be arbitrarily large or small on unstable leaves.
Further, since our definition for unstable metric entropy
has a similar form as the classical one,
some properties can be obtained by following the same strategies,
such as upper semi-continuity and variational principle
(see the proof of Proposition~\ref{Pusc} and Theorem~D).

%There is no existing notion for unstable topologically entropy.
The unstable topological entropy we introduce can be regarded as
a kind of conditional topological entropy.
Conditional topological entropy is first introduced in \cite{Mis1},
where the author uses open covers for condition.
Our definition is close to the topological conditional entropy
defined in \cite{HYZ}, though they have a factor map to provide
a natural partition for condition and we do not.
However, the quantity that gives the same information is the
unstable volume growth introduced by Hua-Saghin-Xia (\cite{HSX}),
reminiscent from the works by Yomdin and Newhouse (\cite{Y,N}).
With the notion they obtained a formula for upper bound of
the metric entropy of an invariant measure using unstable volume growth
and the sum of positive Lyapunov exponents in the center direction
(see \eqref{fineqHSX}).
With unstable entropies we defined, we can give different versions of
the formula in both measure theoretic and topological categories.
For the former one we use unstable metric entropy and the sum of positive
center Lyapunov exponents for upper bounds of metric entropy
(Corollary~A.1).  For the latter one we conclude in Corollary~C.2
that the topological entropy of a partially hyperbolic diffeomorphism
is bounded by the unstable topological entropy and the growth rates
of $\|Df|_{E^c}\|$ and its outer products (defined in \eqref{fDfouter1}
and \eqref{fDfouter2}).
Further we can define a \emph{transversal entropy}
(Definition~\ref{Deftransent})
and then obtain that the upper bound can be given by the sum of the unstable
topological entropy and transversal entropy.

When the paper was being written we found a paper by Jiagang Yang \cite{Yan}
that contains the upper semi-continuity of the unstable metric entropy
with respect to both the invariant measures $\mu$ and the dynamical
systems $f$, by constructing an increasing partition $\xi$.
It is more general than the result in Proposition~\ref{Pusc}.
However, we still give our proof since it is much simpler and straightforward
when our definition of the unstable metric entropy is used.

Though the unstable entropies we introduce here are for the unstable
foliations of partially hyperbolic diffeomorphisms, it is obvious
that they can be applied to more general settings.
If a diffeomorphism has a hierarchy of unstable foliations,
the entropies can be defined on each level as long as the map
is uniformly expanding restricted to the leaves of the foliation.
Also, if partial hyperbolicity holds only on a closed invariant subset
$\Lambda$, such as an Axiom~A system crossing any slow motion system,
then we can study unstable entropies for the system $f|_{\Lambda}$,
the diffeomorphism restricted to $\Lambda$.
%Of course, more challenge problem is for the case that
%the diffeomorphism is nonuniformly expanding restricted to
%the leaves of a foliation.

The paper is organized as following.
In Section~\ref{Sdef} we give definitions of unstable metric and topological
entropy and state the main results.
We prove Theorem~A and provide some properties of unstable metric entropy
in Section~\ref{SUME}.  Section~\ref{SSMBThm} is for a proof of
Shannon-McMillan-Breiman theorem (Theorem~B) for unstable metric entropy.
Properties of unstable topological entropy and proof of Theorem~C are
provided in Section~\ref{SUTE}.  The last section, Section~\ref{SVP},
is for the proof of Theorem~D, the variational principle.

%%%%%%%%%%%%%%%%%%%%%%%%%%%%%%%%%%%%%%%%%%%%%%%
%%%%%%%%%%%%%%%%%%%%%%%%%%%%%%%%%%%%%%%%%%%%%%%
\section{Definitions and statements of main results}\label{Sdef}
\setcounter{equation}{0}
%%%%%%%%%%%%%%%%%%%%%%%%%%%%%%%%%%%%%%%%%%%%%%%
%%%%%%%%%%%%%%%%%%%%%%%%%%%%%%%%%%%%%%%%%%%%%%%

%In this section we briefly describe the definition of unstable metric
%entropy, and leave more details in the next section.
Let $M$ be an $n$-dimensional smooth, connected and compact Riemannian manifold without boundary and $f: M \rightarrow M$ a $C^{1}$-diffeomorphism. $f$ is said to be \emph{partially hyperbolic} (cf. for example \cite{RRU}) if there exists a nontrivial $Tf$-invariant splitting $TM= E^s \oplus E^c \oplus E^u$
of the tangent bundle into stable, center, and unstable distributions, such that all unit vectors $v^{\sigma} \in E_x^\sigma$ ($\sigma= c,s,u$) with $x\in M$ satisfy
\begin{equation*}
\|T_xfv^s\| < \|T_xfv^c\| < \|T_xfv^u\|,
\end{equation*}
and
\begin{equation*}
\|T_xf|_{E^s_x}\| <1 \ \ \ \text{\ and\ \ \ \ } \|T_xf^{-1}|_{E^u_x}\| <1,
\end{equation*}
for some suitable Riemannian metric on $M$. The stable distribution $E^s$ and unstable distribution $E^u$ are integrable to the stable and unstable foliations $W^s$ and $W^u$ respectively such that $TW^s=E^s$ and $TW^u=E^u$ (cf. \cite{HPS}).

In this paper we always assume that $f$ is a $C^1$-partially hyperbolic
diffeomorphism of $M$, and $\mu$ is an $f$-invariant probability measure.

For a partition $\a$ of $M$, let $\a(x)$ denote the element of $\a$
containing $x$.
If $\a$ and $\b$ are two partitions such that $\a(x)\subset \b(x)$
for all $x\in M$, we then write $\a \geq \b$ or $\b\leq \a$.
A partition $\xi$ is \emph{increasing} if $f^{-1}\xi \geq \xi$.
For a measurable partition $\b$, we denote
$\disp\b_m^n=\vee_{i=m}^n f^{-i}\b$.  In particular,
$\disp\b_0^{n-1}=\vee_{i=0}^{n-1} f^{-i}\b$.

Take $\e_0>0$ small.
Let $\P=\P_{\e_0}$ denote the set of finite measurable partitions of $M$
whose elements have diameters smaller than or equal to $\e_0$, that is,
$\diam \a:=\sup\{\diam A: A\in \a\}\le \e_0$.
For each $\b\in \P$ we can define a finer partition $\eta$ such that
$\eta(x)=\b(x)\cap W^u_\loc(x)$ for each $x\in M$, where $W^u_\loc(x)$
denotes the local unstable manifold at $x$ whose size is
greater than the diameter $\e_0$ of $\beta$.
Clearly $\eta$ is a measurable partition satisfying $\eta\ge \b$.
%We denote such partition $\b$ by $\eta^\sharp$, that is, for any
%$\eta^\sharp\in \P$ and $\eta(x)=\eta^\sharp(x)\cap W^u_\loc(x)$ for any
%$x\in M$.
Let $\P^u=\P^u_{\e_0}$ denote the set of partitions $\eta$ obtained this way.

A partition $\xi$ of $M$ is said to be
\emph{subordinate to unstable manifolds} of $f$ with respect
to a measure $\mu$ if for $\mu$-almost every $x$,
$\xi(x)\subset W^u(x)$
and contains an open neighborhood of $x$ in $W^u(x)$.
It is clear that if $\a\in \P$ such that
$\mu(\partial \a)=0$ where $\partial \a:=\cup_{A\in \a} \partial A$,
then the corresponding $\eta$ given by $\eta(x)=\a(x)\cap W^u_\loc(x)$
is a partition subordinate to unstable manifolds of $f$.
%We denote by $\P^u$ the set of partitions in $\P^u_0$ that are
%subordinate to unstable manifolds of $f$.

Given a measure $\mu$ and measurable partitions $\a$ and $\eta$, let
$$H_\mu(\a|\eta):=-\int_M \log \mu_x^\eta(\alpha(x))d\mu(x)$$
denote the conditional entropy of $\a$ given $\eta$ with respect to $\mu$,
where $\{\mu_x^\eta: x\in M\}$ is a family of conditional measures of $\mu$ relative to $\eta$.
The precise meaning is given in Definition~\ref{Defcondent1} in
the next section (see also \cite{R}).

\begin{definition}\label{Defuentropy}
The \emph{conditional entropy of $f$ with respect to a measurable partition $\a$
given $\eta\in \P^u$} is defined as
$$h_\mu(f, \alpha|\eta)
=\limsup_{n\to \infty}\frac{1}{n}H_\mu(\alpha_0^{n-1}|\eta).
$$
The \emph{conditional entropy of $f$ given $\eta\in \P^u$}
is defined as
$$h_\mu(f|\eta)
=\sup_{\alpha \in \P}h_\mu(f, \alpha|\eta).
$$
and the \emph{unstable metric entropy of $f$} is defined as
\[
h_\mu^u(f)=\sup_{\eta\in \P^u}h_\mu(f|\eta).
\]
\end{definition}

\begin{remark}
In the definition of $h_\mu(f, \alpha|\eta)$ we take $\limsup$ instead of
$\lim$, since $\eta$ is not invariant under $f$ and hence the sequence $\{H_\mu(\alpha_0^{n-1}|\eta)\}$ is not necessarily subadditive. Therefore, existence of such a limit is not obvious.
\end{remark}

We show in Lemma~\ref{Lcond} that $h_\mu(f|\eta)$ is independent
of $\eta$, as long as it is in $\P^u$.  Hence, we actually have
$h_\mu^u(f)=h_\mu(f|\eta)$ for any $\eta\in \P^u$.

\medskip
Suppose that $\mu$ is ergodic. Recall a hierarchy of metric entropies
$h_\mu(f, \xi_i):=H_\mu(\xi_i|f\xi_i)$ introduced by Ledrappier and Young
in \cite{LY2},
where $i=1, \cdots, \tilde u$, and $\tilde u$
is the number of distinct positive Lyapunov exponents.
For each $i$, $\xi_i$ is an increasing partition subordinate to
the $i$th level of the unstable leaves $W^{(i)}$, and is a generator.
(See Subsection~{\ref{SSmetric2} for precise meaning.)
It is proved there that $h_{\mu}(f, \xi_{\tilde u})=h_\mu(f)$, the
metric entropy of $\mu$.

If there are $u$ distinct Lyapunov exponents on unstable subbundle,
then the $u$th unstable foliation are the unstable foliation of the
partially hyperbolic system $f$.
We show that the unstable metric entropy we define is identical
to $h_\mu(f, \xi_u)$ given by Ledrappier-Young.

Denote by $\Q^u$ the set of increasing partitions $\xi_u$
that are subordinate to $W^{u}$, and are generators, that is, partitions $\xi_u$ satisfying condition~(i)-(iii)
in Lemma~\ref{Lpartition} in Subsection~\ref{SSmetric2}.

\begin{TheoremA}\label{ThmA}
Suppose $\mu$ is an ergodic measure.
Then for any $\a\in \P$, $\eta\in \P^u$ and $\xi\in {\Q}^u$,
$$h_\mu(f, \a|\eta)=h_\mu(f, \xi).$$
Hence,
$$h_\mu^u(f)=h_\mu(f|\eta)=h_\mu(f, \xi).$$
\end{TheoremA}

It is easy to see the following relation by the definition of
unstable metric entropy and a formula given by Ledrappier and Young.

Let $\{\lambda_i^c\}$ denote distinct Lyapunov exponents of $\mu$
in the center direction, and $m_i$ denote the multiplicity of $\lambda_i^c$.

\begin{CorollaryA1}\label{CA1}
$h_\mu^u(f)\le h_\mu(f)$.

Moreover, if $f$ is $C^{1+\alpha}$, then
$h_\mu(f)\le h_\mu^u(f)+\sum_{\lambda_i^c>0} \lambda_i^c m_i$.
In particular, if there is no positive Lyapunov exponent
in the center direction at $\mu$-a.e. $x\in M$,  then $h_\mu^u(f)= h_\mu(f)$.
\end{CorollaryA1}

In \cite{HSX} the authors proved that for any ergodic measure $\mu$,
\begin{equation}\label{fineqHSX}
h_\mu(f)\leq \chi^u(f)+\sum_{\lambda_i^c>0}\lambda_i^c m_i,
\end{equation}
where $\chi^u(f)$ denotes the volume growth of the unstable foliation
(see \eqref{fvolgrowth1} and \eqref{fvolgrowth2} below for precise meaning).
The part of the inequality for the upper bound of $h_\mu(f)$
in the corollary can be regarded as a version of \eqref{fineqHSX}
in measure theoretic category.

\medskip
The first equation of Theorem A gives that $h_\mu(f, \a|\eta)$
is independent of $\a$.
The proof of the theorem also gives the following:

\begin{CorollaryA2}\label{CA2}
$\disp h^u_\mu(f)=h_\mu(f, \a|\eta)
=\lim_{n\to\infty}\frac{1}{n}H_\mu(\a_0^{n-1}|\eta)$
for any $\a\in \P$ and $\eta\in \P^u$.
\end{CorollaryA2}

The next result is a version of Shannon-McMillan-Breiman theorem
for the unstable metric entropy. In the proof of Theorem~A we actually showed that the sequence of the integrals of functions $\disp\Big\{\frac{1}{n}I_\mu(\alpha_0^{n-1}|\eta)\Big\}$ converges to
$h_\mu(f, \a|\eta)$.
This theorem states that the functions converge almost everywhere.

\begin{TheoremB}\label{ThmB}
Suppose $\mu$ is an ergodic measure of $f$.  Let $\eta\in \P^u$ be given.
Then for any partition $\a$ with $H_\mu(\a|\eta)<\infty$, we have
\begin{equation*}
\lim_{n\to \infty}\frac{1}{n}I_\mu(\alpha_0^{n-1}|\eta)(x)
=h_\mu(f, \a|\eta) \quad \quad \mu\text{-a.e.} x\in M.
\end{equation*}
\end{TheoremB}

We can use $\xi\in \Q^u$ for the given partition as well.

\begin{CorollaryB1}\label{PSMB}
Let $\mu$ be $f$-ergodic and $\xi\in \Q^u$.
Then for any partition $\a$ with $H_\mu(\a|\xi)<\infty$, we have
\begin{equation*}
\lim_{n\to \infty}\frac{1}{n}I_\mu(\alpha_0^{n-1}|\xi)(x)
=h_\mu(f, \a|\xi) \quad \quad \mu\text{-a.e.} x\in M,
\end{equation*}
where $h_\mu(f, \a|\xi)$ is defined as in
Definition~\ref{Defuentropy} with $\eta$ replaced by $\xi$.
\end{CorollaryB1}

\begin{remark}
We mention here that in Lemma~\ref{LSMB1} we also obtain
\[
\disp \lim_{n\to \infty}\frac{1}{n}I_\mu(\xi_0^{n-1}|\xi)(x)
=h_\mu(f, \xi)
\]
when $\mu$ is ergodic.
\end{remark}

Now we start to define the unstable topological entropy.

We denote by $d^u$ the metric induced by the Riemannian structure
on the unstable manifold and let
$d^u_{n}(x,y)=\max _{0 \leq j \leq n-1}d^u(f^j(x),f^j(y))$.
Let $W^u(x,\delta)$ be the open ball inside $W^u(x)$ centered at $x$
of radius $\delta$ with respect to the metric $d^u$.
Let $N^u(f,\epsilon,n,x,\delta)$ be the maximal
number of points in $\overline{W^u(x,\delta)}$ with pairwise
$d^u_{n}$-distances at least $\epsilon$.  We call such set an
\emph{$(n,\epsilon)$ u-separated set} of $\overline{W^u(x,\delta)}$.

\begin{definition}\label{Defutopent1}
The \emph{unstable topological entropy} of $f$ on $M$ is defined by
\begin{equation*}
%\begin{aligned}
h^u_{\text{top}}(f)
=\lim_{\delta \to 0}\sup_{x\in M}h^u_{\text{top}}(f, \overline{W^u(x,\delta)}),
%\end{aligned}
\end{equation*}
where
\begin{equation*}
%\begin{aligned}
 h^u_{\text{top}}(f, \overline{W^u(x,\delta)})
=\lim_{\epsilon \to 0}\limsup_{n\to \infty}\frac{1}{n}\log N^u(f,\epsilon,n,x,\delta).
%\end{aligned}
\end{equation*}
\end{definition}

We can also define unstable topological entropy by using
$(n,\epsilon)$ u-spanning sets or open covers to get equivalent
definitions.

Unstable topological entropy defined here can be regarded
as the asymptotic rate of orbit divergence along unstable manifolds.
Since $f$ is expanding restricted to unstable manifolds,
the rate of orbit divergence can also be reflected by
the asymptotic rate of the volume growth of unstable manifolds
under iterations of $f$.  Volume growth was
first used by Yomdin and Newhouse for the entropy of
diffeomorphisms (cf. \cite{Y}, \cite{N}).
The unstable volume growth for partially hyperbolic systems
is used in \cite{HSX}, which is defined as following:
\begin{equation}\label{fvolgrowth1}
\chi_u(f)= \sup_{x \in M} \chi_u(x, \delta)
\end{equation}
where
\begin{equation}\label{fvolgrowth2}
\chi_u(x, \delta) = \limsup_{n \to \infty} \frac{1}{n} \log
(\text{Vol} (f^n(W^u(x,\delta))).
\end{equation}

Note that the unstable volume growth is independent of
$\delta$ and the Riemannian metric (cf. Lemma 1.1 in \cite{SX}).
%the definition can be extended for any invariant foliation with $C^1$ leaves (not only the unstable ones).
We show that the unstable topological entropy actually coincides
with the unstable volume growth.

\begin{TheoremC}\label{ThmC}
$h_{\text{top}}^u(f) ={\chi_u(f)}$.
\end{TheoremC}

By the definition and the equality, we have the following facts.

\begin{CorollaryC1}\label{CC1}
$h^u_\top(f)\leq h_\top(f)$.

The equation holds if there is no positive Lyapunov exponent in the
center direction at $\nu$-a.e. with respect to any ergodic measure $\nu$.
\end{CorollaryC1}

We can also give a version of the inequality formula \eqref{fineqHSX}
in terms of unstable topological entropy and growth rates
of $\big\|{\bigwedge}^i Df^n|_{E^c}\big\|$ in the center direction.
For $1\le i\le \dim E^c$, let
\begin{equation}\label{fDfouter1}
\sigma^{(i)}
=\lim_{n\to\infty}\frac{1}{n}\log\Big\|{\bigwedge}^i Df^n|_{E^c}\Big\|,
\end{equation}
where ${\bigwedge}^i Df^n|_{E^c}$ denotes the $i$th outer product
of the differential $Df^n|_{E^c}$.  The limit exists because
of subadditivity of $\log\|{\bigwedge}^i Df^n|_{E^c}\|$.
Then we denote
\begin{equation}\label{fDfouter2}
\sigma=\max\{\sigma^{(i)}: \ i=1, \cdots, \dim E^c\}.
\end{equation}
Note that $\sigma^{(i)}$ is greater than the sum of the largest $i$ Lyapunov
exponents in the center direction at any point $x$ whenever they exist.

%Now we can give a version of the inequality \eqref{fineqHSX}
%not to involve measure theoretic notions.

\begin{CorollaryC2}\label{CC2}
$h_\top(f)\leq h_\top^u(f)+\sigma$.

The equation holds if $\sigma^{(1)}\le 0$.
\end{CorollaryC2}

Similar to the quantity $h_{i+1}-h_{i}$, the difference between
consecutive hierarchy entropies, used by Ledrappier and Young
in \cite{LY2},
we can define transversal topological entropy as following.

Let $N(f,\epsilon,n,x,\delta)$ be the maximal number of points
in $\overline{B(x,\delta)}$ with pairwise $d_{n}$-distances
at least $\epsilon$, where $B(x,\delta)$ denotes the open ball about $x$
of radius $\d$, and $d_{n}(x,y)=\max _{0\leq j\leq n-1}d(f^j(x),f^j(y))$.

\begin{definition}\label{Deftransent}
The \emph{transversal topological entropy} of $f$ on $M$ is defined by
\begin{equation*}
h^t_{\top}(f)
=\lim_{\delta \to 0}\sup_{x\in M}h^t_{\text{top}}(f, \overline{B(x,\delta)}),
\end{equation*}
where
\begin{equation*}
h^t_{\text{top}}(f, \overline{B(x,\delta)})
=\lim_{\epsilon \to 0}\limsup_{n\to \infty}\frac{1}{n}
\big[\log N(f,\epsilon,n,x,\delta)-\log N^u(f,\epsilon,n,x,\delta)\big].
\end{equation*}
\end{definition}
With the notion we can give another version of formula in \cite{HSX}
in topological category.

\begin{CorollaryC3}\label{CC3}
$h_\top(f)\le h_\top^u(f)+h^t_\top(f)$.
\end{CorollaryC3}

We mention here that in \cite{WZ} the authors proved that if
$f$ is a partially hyperbolic diffeomorphism with a uniformly compact
center foliation, then $h_\top(f) \le p^c(f) + h_\top(f, {\mathcal W}^c)$,
where $p^c(f)$ is the growth rate of periodic center leaves, i.e.,
the leaves $W^c(x)$ with $f^nW^c(x)=W^c(f^nx)$, and
$h_\top(f, {\mathcal W}^c)$ is given by the growth rate of
$(n,\epsilon)$ separated sets on center leaves.

For partially hyperbolic diffeomorphisms, we can also build
a variational principle for unstable metric entropy
and unstable topological entropy.

Let $\mathcal{M}_f(M)$ and $\mathcal{M}^e_f(M)$ denote
the set of all $f$-invariant and ergodic probability measures on $M$
respectively.

\begin{TheoremD}\label{ThmD}
Let $f: M \to M$ be a $C^1$-partially hyperbolic diffeomorphism. Then
$$h^u_{\text{top}}(f)=\sup\{h_{\mu}^u(f): \mu \in \mathcal{M}_f(M)\}.
$$
Moreover,
$$
h^u_{\text{top}}(f)=\sup\{h_{\nu}^u(f): \nu \in \mathcal{M}^e_f(M)\}.
$$
\end{TheoremD}

\medskip
%%%%%%%%%%%%%%%%%%%%%%%%%%%%%%%%%%%%%%%%%%%%%%%
%%%%%%%%%%%%%%%%%%%%%%%%%%%%%%%%%%%%%%%%%%%%%%%
\section{Unstable metric entropy}\label{SUME}
\setcounter{equation}{0}
%%%%%%%%%%%%%%%%%%%%%%%%%%%%%%%%%%%%%%%%%%%%%%%
%%%%%%%%%%%%%%%%%%%%%%%%%%%%%%%%%%%%%%%%%%%%%%%

\subsection{Conditional entropy}
In this subsection we provide more detailed information about
conditional entropy and some properties as a supplement
of Definition~\ref{Defuentropy}.

Recall that for a measurable partition $\eta$ of a measure space $X$
and a probability measure $\nu$ on $X$, the \emph{canonical system
of conditional measures for $\nu$ and $\eta$} is a family of probability
measures $\{\nu_x^\eta: x\in X\}$ with $\nu_x^\eta\bigl(\eta(x)\bigr)=1$,
such that for every measurable set $B\subset X$, $x\mapsto \nu_x^{\eta}(B)$
is measurable and
\[
\nu (B)=\int_X\nu_x^{\eta}(B)d\nu(x).
\]
(See e.g. \cite{R} for reference.)

The following notions are standard.

\begin{definition}\label{Defcondent1}
The \emph{information function of $\alpha\in \P$} are defined as
$$I_\mu(\alpha)(x):=-\log \mu(\alpha(x)),$$
and the \emph{entropy of partition $\alpha$} as
$$H_\mu(\alpha):=\int_M I_\mu(\alpha)(x)d\mu(x)=-\int_M \log \mu(\alpha(x))d\mu(x).$$
The \emph{conditional information function of $\alpha\in \P$
with respect to a measurable partition $\eta$ of $M$} is defined as
$$I_\mu(\alpha|\eta)(x):=-\log \mu_x^\eta(\alpha(x)).$$
Then the \emph{conditional entropy of $\alpha$ with respect to $\eta$}
is defined as
$$H_\mu(\alpha|\eta):=\int_M I_\mu(\alpha|\eta)(x)d\mu(x)=-\int_M \log \mu_x^\eta(\alpha(x))d\mu(x).$$
\end{definition}

The properties in the following lemma is well known (see e.g. \cite{R}).

\begin{lemma}\label{Lcond1}
Let $\a$, $\b$ and $\c$ be measurable partitions with
$ H_\mu(\a|\c), H_\mu(\b|\c)<\infty$.
\begin{enumerate}
\item[(i)]
If $\a\le \b$, then $I_\mu(\a|\c)(x)\le I_\mu(\b|\c)(x)$ and
$H_\mu(\a|\c)\le H_\mu(\b|\c)$.

\item[(ii)]
$I_\mu(\a\vee\b|\c)(x)=I_\mu(\a|\c)(x)+I_\mu(\b|\a\vee\c)(x)$ and
%for almost every $x$. Hence,
$H_\mu(\a\vee\b|\c)=H_\mu(\a|\c)+H_\mu(\b|\a\vee\c)$.

\item[(iii)]
$H_\mu(\a\vee\b|\c)\le H_\mu(\a|\c)+H_\mu(\b|\a)$.

\item[(iv)]
$H_\mu(\b|\c)\le H_\mu(\a|\c)+H_\mu(\b|\a)$.

\item[(v)]
If $\b\le \c$, then $H_\mu(\a|\b)\ge H_\mu(\a|\c)$.
\end{enumerate}
\end{lemma}

\begin{remark}
We mention here that  $\b\le \c$ does not imply
$I_\mu(\a|\b)(x)\ge I_\mu(\a|\c)(x)$ for $\mu$-a.e. $x$,
though we have (v) in the above lemma.
\end{remark}

Recall that for a probability space $(X, {\mathcal B}, \nu)$ and
a sequence of increasing sub-$\sigma$-algebras
${\mathcal B}_1\subset {\mathcal B}_2\subset \dots
\subset {\mathcal B}$, a sequence of functions $\{\phi_n\}$
is a \emph{martingale} with respect to $\{{\mathcal B}_n\}$ if
\begin{enumerate}
\item[(i)] $\phi_n$ is ${\mathcal B}_n$ measurable for all $n>0$; and
\item[(ii)] $E_\nu(\phi_{n+1}|{\mathcal B}_n)= \phi_n$\ $\nu$-a.e. $x$,
where $E_\nu$ denotes the expectation.
\end{enumerate}
If ``$=$'' in Condition (ii) is replaced by ``$\le$'' or ``$\ge$'',
then the sequence $\{\phi_n\}$ is called a \emph{supermartingale}
or \emph{submartingale} respectively.

A supermartingale $\{\phi_n\}$ is $L^1$ bounded if
$\sup_n E_\nu(|\phi_n|)<\infty$.

Note that if $\{\phi_n\}$ is a supermartingale, then
$\{-\phi_n\}$ is a submartingale.  $\{\phi_n\}$ is $L^1$ bounded
if and only if $\{-\phi_n\}$ is $L^1$ bounded.  So
Doob's martingale convergence theorem can be stated in the following way.

\begin{theorems}[Doob's martingale convergence theorem]
Every $L^1$ bounded supermartingale or submartingale $\{\phi_n\}$
converges almost everywhere.
\end{theorems}

Since a martingale is also a supermartingale, the theorem
gives convergence of $L^1$ bounded martingales.

\begin{lemma}\label{Lphi1}
Let $\a\in \P$ and $\{\zeta_n\}$ be a sequence of increasing measurable
partitions with $\zeta_n\nearrow \zeta$.
Then for $\phi_n(x)=I_\mu(\a|\zeta_n)(x)$,
$\phi^*: =\sup_{n}\phi_n \in L^1(\mu)$.
\end{lemma}

The proof of the lemma can be found in the proof of Lemma~14.27
in~\cite{Gl} or Lemma~2.1 and Corollary~2.2 on p.261 in \cite{Pet}.

For partitions $\{\zeta_n\}$ and $\zeta$ given in the last lemma,
let $\{\B(\zeta_n)\}$ and $\B(\zeta)$ be the sub-$\sigma$-algebras
generated by $\{\zeta_n\}$ and $\zeta$ respectively, that is,
$\{\B(\zeta_n)\}$ is the smallest sub-$\sigma$-algebra containing
elements of $\zeta_n$.
Let $\phi_n(x)=I_\mu(\a|\zeta_n)(x)=-\log \mu^{\zeta_n}_x(\a(x))$.
Then $\int \phi_n d\mu=H_\mu(\a|\zeta_n)$.
It is well known by Jensen's inequality that $\{\phi_n\}$ is
a supermartingale.

Using Doob's martingale convergence theorem, we know that
$\phi_n=I_\mu(\a|\zeta_n)$ converges to $I_\mu(\a|\zeta)$ almost everywhere.
Then Lemma~\ref{Lphi1} gives that the sequence $\{\phi_n\}$
is bounded by a $L^1$ function $\phi^*$.  So Lebesgue's dominated
convergence theorem gives convergence of $H_\mu(\a|\zeta_n)$ to
$H_\mu(\a|\zeta)$.  Hence the following lemma is established
(cf. Theorem 14.28 in \cite{Gl}).

\begin{lemma}\label{Lcondzetan}
Let $\a\in \P$ and $\{\zeta_n\}$ be a sequence of increasing measurable
partitions with $\zeta_n\nearrow \zeta$.   Then
\begin{enumerate}
  \item[(i)] $\lim_{n\to\infty}I_\mu(\a|\zeta_n)(x)=I_\mu(\a|\zeta)(x)$
for $\mu$-a.e. $x$;  and
  \item[(ii)] $\lim_{n\to\infty}H_\mu(\a|\zeta_n)=H_\mu(\a|\zeta)$.
\end{enumerate}
\end{lemma}

\begin{lemma}\label{Linduction}
Suppose $\a$, $\b$ and $\c$ are measurable partitions.
\begin{enumerate}
\item[(i)]
$I_\mu(\b_0^{n-1}|\c)(x)=I_\mu(\b|\c)(x)
+\sum_{i=1}^{n-1}I_\mu(\b|f^i(\b_0^{i-1}\vee \c))(f^i(x))$. Hence it follows
$H_\mu(\b_0^{n-1}|\c)=H_\mu(\b|\c)+\sum_{i=1}^{n-1}H_\mu(\b|f^i(\b_0^{i-1}\vee \c))$.

\item[(ii)]
$I_\mu(\alpha_0^{n-1}|\c)(x)=I_\mu(\alpha|f^{n-1}\c)(f^{n-1}(x))
+\sum_{i=0}^{n-2}I_\mu(\alpha|\alpha_1^{n-1-i}\vee f^i\c)(f^i(x))$.
Hence, $H_\mu(\alpha_0^{n-1}|\c)
=H_\mu(\alpha|f^{n-1}\c)+\sum_{i=0}^{n-2}H_\mu(\alpha|\alpha_1^{n-1-i}\vee f^i\c)$.
\end{enumerate}
\end{lemma}

\begin{proof}
(i) Replacing $\a$ and $\b$ by $\b_0^{i-1}$ and $f^{-i}\b$
in Lemma~\ref{Lcond1}(ii) respectively, we have
\begin{equation*}
\begin{split}
 I_\mu(\b_0^{i}|\c)(x)
=&I_\mu(\b_0^{i-1}|\c)(x)+I_\mu(f^{-i}\b|\b_0^{i-1}\vee\c)(x) \\
=&I_\mu(\b_0^{i-1}|\c)(x)+I_\mu(\b|f^i(\b_0^{i-1}\vee \c))(f^i(x)).
\end{split}
\end{equation*}
Summing the equality over $i$ from $1$ to $n-1$ we get
the first equality in part (i).
The second one follows by integrating the first equality.

(ii) Replacing $\a$ and $\b$ by $\a_1^{n-1}$ and $\a$ in Lemma~\ref{Lcond1}(ii) respectively,
we have
\begin{equation*}
\begin{aligned}
I_\mu(\alpha_0^{n-1}|\c)(x)=&I_\mu(\alpha_1^{n-1}|\c)(x)
  +I_\mu(\alpha|\alpha_1^{n-1}\vee\c)(x)\\
=&I_\mu(\alpha_0^{n-2}|f\c)(f(x))+I_\mu(\alpha|\alpha_1^{n-1}\vee\c)(x).
\end{aligned}
\end{equation*}
By induction, we have
$$I_\mu(\alpha_0^{n-1}|\c)(x)=I_\mu(\alpha|f^{n-1}\c)(f^{n-1}(x))
+\sum_{i=0}^{n-2}I_\mu(\alpha|\alpha_1^{n-1-i}\vee f^i\c)(f^i(x)).$$
Integrating both sides of the formula, we have the second equality of the part.
\end{proof}

%%%%%%%%%%%%%%%%%%%%%%%%%

\begin{lemma}\label{Lcond2}
\begin{enumerate}
\item[(i)]
For any $\eta_1,\eta_2\in \P^u$,
$H_\mu(\eta_2|\eta_1), H_\mu(\eta_1|\eta_2)< \infty$.  Hence
\begin{equation*}
\lim_{n\to \infty}\frac{1}{n}H_\mu(\eta_2|\eta_1)
=0
=\lim_{n\to \infty}\frac{1}{n}H_\mu(\eta_1|\eta_2).
\end{equation*}

\item[(ii)]
For any $\a,\b\in \P$ and $\eta\in \P^u$,
\begin{equation*}
\lim_{n\to \infty}\frac{1}{n}H_\mu(\a_0^{n-1}|\b_0^{n-1}\vee\eta)
=0
=\lim_{n\to \infty}\frac{1}{n}H_\mu(\b_0^{n-1}|\a_0^{n-1}\vee\eta).
\end{equation*}
\end{enumerate}
\end{lemma}

\begin{proof}
(i) Recall that $\P=\P_{\e_0}$ is the set of finite measurable partitions
of diameter less than or equal to $\e_0$.  For any $\eta_1, \eta_2\in \P^u$,
there exist $\alpha_1, \alpha_2\in \P$ such that
$\eta_i(x)=\a_i(x)\cap W^u_\loc(x)$, $i=1,2$, for all $x\in M$.
Let $N_i$ be the cardinality of $\alpha_i$.
Then for any $x\in M$, $\eta_1(x)$ intersects at most $N_2$ elements
of $\a_2$, and therefore intersects at most $N_2$ elements of $\eta_2$.
Similarly, $\eta_2(x)$ intersects at most $N_1$ elements
of $\eta_1$.  So we have $H_\mu(\eta_2|\eta_1)\le \log N_2$ and
$H_\mu(\eta_1|\eta_2)\le\log N_1$.

(ii) Applying Lemma~\ref{Linduction}(ii) with $\c=\b_0^{n-1}\vee \eta$,
and using the fact $f^i\b_0^{n-1}=\b_{-i}^{n-i-1}\ge \b_{1}^{n-i-1}$, we have
\begin{equation*}%\label{fLcond2a}
\begin{split}
H_\mu(\a_0^{n-1}|\b_0^{n-1}\vee\eta)
\le&H_\mu(\a|f^{n-1}\b_0^{n-1}\vee f^{n-1}\eta)
+\sum_{i=0}^{n-2}H_\mu(\a|\a_1^{n-1-i}\vee \b_{1}^{n-i-1}\vee f^i\eta) \\
\le &H_\mu(\a|\b\vee f^{n-1}\eta)
+\sum_{i=1}^{n-1}H_\mu(\a|\a_1^{i}\vee \b_{1}^{i}\vee f^{n-i-1}\eta) \\
\end{split}
\end{equation*}
Since for any $x$,
$(\a_1^{n}\vee \b_{1}^{n}\vee \eta)(x)\subset W^u_\loc(x)$ and
$\diam (\a_1^{n}\vee \b_{1}^{n}\vee \eta)(x)\to 0$ as $n\to \infty$,
we have
$\disp \lim_{n\to \infty}H_\mu(\a|\a_1^{n}\vee \b_{1}^{n}\vee \eta)=0$ by Lemma \ref{Lcondzetan}(ii).
It means that the terms in the summation in the last inequality tend
to $0$.  Hence, we get that
\[
\lim_{n\to \infty}\frac{1}{n}H_\mu(\a_0^{n-1}|\b_0^{n-1}\vee\eta)=0.  \qedhere
\]
\end{proof}

%We shall prove the following lemma.
\begin{lemma}\label{Lcond}
\begin{enumerate}
\item[(i)]
For any $\a\in \P$ and $\eta_1,\eta_2\in \P^u$,
$h_\mu(f, \a|\eta_1)=h_\mu(f, \a|\eta_2)$.

\item[(ii)]
For any $\a,\b\in \P$ and $\eta\in \P^u$,
\begin{equation*}%\label{fcond}
\begin{split}
\limsup_{n\to \infty}\frac{1}{n}H_\mu(\alpha_0^{n-1}|\eta)
=&\limsup_{n\to \infty}\frac{1}{n}H_\mu(\beta_0^{n-1}|\eta),  \\
\liminf_{n\to \infty}\frac{1}{n}H_\mu(\alpha_0^{n-1}|\eta)
=&\liminf_{n\to \infty}\frac{1}{n}H_\mu(\beta_0^{n-1}|\eta).
\end{split}
\end{equation*}
Hence, $h_\mu(f, \a|\eta)=h_\mu(f, \b|\eta)$.
\end{enumerate}
\end{lemma}

\begin{proof}
(i) By Lemma~\ref{Lcond1}(iv) we have
\begin{equation*}%\label{e:compute}
\begin{split}
H_\mu(\alpha_0^{n-1}|\eta_1)
\le H_\mu(\alpha_0^{n-1}|\eta_2)+H_\mu(\eta_2|\eta_1), \\
H_\mu(\alpha_0^{n-1}|\eta_2)
\le H_\mu(\alpha_0^{n-1}|\eta_1)+H_\mu(\eta_1|\eta_2).
\end{split}
\end{equation*}
Hence by using Lemma~\ref{Lcond2}(i) we get
\[
\limsup_{n\to \infty}\frac{1}{n}H_\mu(\alpha_0^{n-1}|\eta_1)
=\limsup_{n\to \infty}\frac{1}{n}H_\mu(\alpha_0^{n-1}|\eta_2).
\]
Then the result of part (i) of the lemma follows.

(ii) Similarly by Lemma~\ref{Lcond1}(i) and (ii) we have
\begin{equation*}%\label{e:compute}
\begin{split}
H_\mu(\a_0^{n-1}|\eta)
\le H_\mu(\b_0^{n-1}|\eta)+H_\mu(\a_0^{n-1}|\b_0^{n-1}\vee \eta), \\
H_\mu(\b_0^{n-1}|\eta)
\le H_\mu(\a_0^{n-1}|\eta)+H_\mu(\b_0^{n-1}|\a_0^{n-1}\vee \eta).
\end{split}
\end{equation*}
By dividing the inequalities by $n$, and taking $\limsup$ and $\liminf$,
and then by Lemma~\ref{Lcond2}(ii) we get equalities of the lemma.
\end{proof}

By this lemma, $h_\mu(f, \a|\eta)$ is independent of $\a$ and $\eta$
as long as $\a\in \P$ and $\eta\in \P^u$.
Hence we can define the unstable metric entropy $h_\mu^u(f)=h_\mu(f, \a|\eta)$
for any $\a\in \P$ and $\eta\in \P^u$.

%\newpage
%%%%%%%%%%%%%%%%%%%%%%%%%%%%%%%%%%%%%%%
%%%%%%%%%%%%%%%%%%%%%%%%%%%%%%%%%%%%%%%
\subsection{Increasing partitions $\xi_u$}\label{SSmetric2}

For an ergodic measure $\mu$ with positive Lyapunov exponents
$\l_1>\l_2>\cdots >\l_{\tilde u}>0$, let
$E^{(1)}\subset E^{(2)} \subset \cdots \subset E^{(\tilde u)}$ denote
the subbundles in the tangent bundle consisting of vectors whose Lyapunov
exponents are greater than or equal to $\l_1, \l_2, \cdots, \l_{\tilde u}$
respectively.  It is well known that if $f$ is $C^{1+\alpha}$, then for almost every $x$
there exist unstable manifolds
$W^{(1)}(x)\subset W^{(2)}(x) \subset \cdots \subset W^{(\tilde u)}(x)$
such that if $y\in W^{(i)}(x)$, then
$\disp \limsup_{n\to \infty}-\frac{1}{n}\log d(f^{-n}y, f^{-n}x)\le -\l_i$ for any $1\leq i\leq \tilde u$.
The entropies $h_\mu(f, \xi_i)$ are determined by a hierarchy of partitions
given in the next lemma.

We mention that a partition $\b$ of $M$ is a generator if
$\bigvee_{n=1}^\infty f^{-n}\b=\varepsilon$ where $\varepsilon$
is a partition of $M$ into points up to a set of zero measure.

\begin{lemma}[Lemma 9.1.1 in \cite{LY2}]\label{Lpartition}
Assume that $f$ is $C^{1+\alpha}$. Then there exist measurable partitions
$\xi_1\ge \xi_2\ge \cdots \ge \xi_{\tilde u}$ on $M$ such that
for each $1\le i\le {\tilde u}$,
\begin{enumerate}
  \item[(i)] $\xi_i$ is subordinate to $W^{(i)}$,
  \item[(ii)] $\xi_i$ is increasing,
  \item[(iii)] $\xi_i$ is a generator.
\end{enumerate}
\end{lemma}

To construct such a partition, the authors in \cite{LY2}
first take a point $z$ and then take
\begin{equation}\label{fdefS}
S_i(z,r)= \bigcup_{y\in W(z,r)} W^{(i)}(y, r)
\end{equation}
where $W(z,r)$ is an open ball of radius $r$ centered at $z$
inside a local manifold $W$ passing through $z$ transversally
to the unstable foliation $W^{(i)}$, and $W^{(i)}(y, r)$ are local
unstable manifolds.
Moreover, $z$ and $S_i(z,r)$ are taken in such a way that $\mu(S_i(z,r))>0$
for any $r>0$.
Then define a partition $\hat\xi_{i,z}$ such that
$\hat\xi_{i,z}(y)=W^{(i)}(\bar{y}, r)$ if $y\in S_i(z,r)$,
where $\bar{y}\in W(z,r)$ and $y\in W^{(i)}(\bar{y},r)$, and
$\hat\xi_{i,z}(y)=M\setminus S_i(z,r)$ otherwise.
Next take $\xi_i=\xi_{i,z}:=\vee _{j\geq 0}f^j\hat{\xi}_{i,z}$.
It has been proven (see e.g. \cite{LY2}) that if $\mu$ is ergodic,
then for almost every
small real number $r>0$ in the sense of Lebesgue measure,
$\xi_i$ is subordinate to unstable manifolds $W^{(i)}$ and therefore
is a partition satisfying Lemma~\ref{Lpartition}.

The above construction of such partitions is also carried out in \cite{LS} and \cite{LY1}. Though by the construction it is unclear whether the diameter
of $\xi_i(x)$ is bounded above or below in the metric $d^i$,
the Riemannian metric restricted to $W^{(i)}$,
it has been proven in \cite{LY2} that
$$H_\mu(f^{-1}\xi_i|\xi_i)
=-\int_M \log \mu_x^{\xi_i}\left((f^{-1}\xi_i)(x)\right)d\mu(x)
$$
is finite.
It is also proved that $h_\mu(f, \xi_i):=H_\mu(\xi_i|f\xi_i)=H_\mu(f^{-1}\xi_i|\xi_i)$
is independent of the choice of $\xi_i$ as long as $\xi_i$
satisfies the above conditions (cf. Subsection~(3.1) in \cite{LY1}).

If there are $u$ distinct Lyapunov exponents on unstable subbundle, then the
$u$th unstable foliation are the unstable foliation of the
partially hyperbolic system $f$.
Recall that $\Q^u$ denote the set of partitions $\xi_u$
satisfying (i)-(iii) above with $i=u$. The above construction of $\xi_u$ still applies
even if $f$ is only assumed to be $C^1$, since the unstable foliation of $f$ always exists under $C^1$ regularity.

Denote $S=S_u(z,r)$, the set given in \eqref{fdefS}.
Recall that $\hat\xi_{i,z}$ is a partition defined above
such that $\xi_i=\xi_{i,z}=\vee _{j\geq 0}f^j\hat{\xi}_{i,z}$.
For $i=u$ we denote $\hat\xi=\hat\xi_{u,z}$.
Recall by the notation we use,
\begin{equation*}\label{fdefxik}
\hat\xi_{-k}^{0}=\vee_{j= 0}^{k} f^j\hat{\xi}_{u,z}.
\end{equation*}
We further denote $\hat\xi_{-k}=\hat\xi_{-k}^{0}$.
Hence $\xi=\hat\xi_{-\infty}$.

\begin{lemma}\label{Lepsilonest}
Suppose $\mu$ is an ergodic measure and $\a\in \P$.
For any $\e>0$, there exists $K>0$ such that for any $k\ge K$,
\[
\limsup_{n\to \infty}H_\mu(\alpha|\alpha_1^{n}\vee \hat\xi_{-k}^{n}) \le \e.
\]
\end{lemma}

\begin{proof}
Denote $S_{-k}=\cup_{i=0}^{k}f^iS$, where $S=S_u(z,r)$ is given by \eqref{fdefS}.

Let $\e>0$.  Since $\mu$ is ergodic, $\mu S_{-k}\to 1$ as $k\to \infty$.
So there exists $K>0$ such that for any $k\ge K$,
$\mu(M\setminus S_{-k})\le \e/log N_a$, where $N_\a$ is the cardinality of
the partition $\a$.

Write
\[
H_\mu(\alpha|\alpha_1^{n}\vee \hat\xi_{-k}^{n})
=\int_{S_{-k}}I_\mu (\alpha|\alpha_1^{n}\vee \hat\xi_{-k}^{n})d\mu(x)
+\int_{M\setminus S_{-k}}I_\mu (\alpha|\alpha_1^{n}\vee \hat\xi_{-k}^{n})d\mu(x).
\]

For $x\in S_{-k}$,
$(\alpha_1^{n}\vee \hat\xi_{-k}^{n})(x)\subset W^u_\loc(x)$.
Hence for almost every $x\in S_{-k}$, there exists $N=N(x)>0$ such that
for any $n\ge N$, $(\alpha_1^{n}\vee \hat\xi_{-k}^{n})(x)\subset \a(x)$
and therefore $\log \mu_x^{\alpha_1^{n}\vee \hat\xi_{-k}^{n}}(\a(x))=0$.
Lemma~\ref{Lphi1} with $\zeta_n=\alpha_1^{n}\vee \hat\xi_{-k}^{n}$
implies that Lebesgue dominated convergence theorem can be applied
to integration over $S_{-k}$.  So it follows
\[
\limsup_{n\to \infty}
\int_{S_{-k}}I_\mu (\alpha|\alpha_1^{n}\vee \hat\xi_{-k}^{n})d\mu(x) =0.
\]

For $x\in M\setminus S_{-k}$,
$(\alpha_1^{n}\vee \hat\xi_{-k}^{n})(x)\subset S_{-k}$.
We know that on
$(\alpha_1^{n}\vee \hat\xi_{-k}^{n})(x)$,
\[
\int_{\alpha_1^{n}\vee \hat\xi_{-k}^{n}(x)}
-\log \mu_x^{\alpha_1^{n}\vee \hat\xi_{-k}^{n}}(\a(y))d\mu_x^{\alpha_1^{n}\vee \hat\xi_{-k}^{n}}(y)
\le \log N_\a.
\]
It gives that
\[
\int_{M\setminus S_{-k}}
-\log \mu_x^{\alpha_1^{n}\vee \hat\xi_{-k}^{n}}(\a(x))d\mu(x)
\le \mu(M\setminus S_{-k}) \cdot \log N_\a
\le \e.
\]
So the result of the lemma follows.
\end{proof}

\begin{lemma}\label{Lconvergence}
Let $\mu$ be an ergodic measure.
Suppose $\eta\in \P^u$ that is subordinate to the unstable manifolds,
and $\hat\xi_{-k}$ is a partition described as above,
where $k\in {\mathbb N}\cup\{\infty\}$.
Then for almost every $x$, there is $N=N(x)>0$ such that for any $i>N$,
\[
(\hat\xi_{-k-i}\vee f^i\eta)(f^i(x))=(\hat\xi_{-k-i})(f^i(x)).
\]
Hence, for any partition $\b$ with
$H_\mu(\b|\hat\xi_{-k})<\infty$,
$$
I_\mu(\b|\hat\xi_{-k-i}\vee f^i\eta)(f^i(x))= I_\mu(\b|\hat\xi_{-k-i})(f^i(x))
$$
and therefore
$$
\lim_{i\to \infty}H_\mu(\b|\hat\xi_{-k-i}\vee f^i\eta)=H_\mu(\b|\xi).
$$

In particular, if we take $k=\infty$, then the last two
equalities become
%Then for almost every $x$, there is $N>0$ such that for any $i>N$,
%\[
%(f\xi\vee f^i\eta)(f^i(x))=(f\xi)(f^i(x)).
%\]
%Hence,
$$
I_\mu(\b|\xi\vee f^i\eta)(f^i(x))= I_\mu(\b|\xi)(f^i(x))
$$
and %therefore
$$
\lim_{i\to \infty}H_\mu(\b|\xi\vee f^i\eta)=H_\mu(\b|\xi).
$$
\end{lemma}

\begin{proof}
Since $\eta$ is subordinate to $W^u$, for $\mu$-a.e. $x$,
there is $r=r(x)>0$ such that $B^u(x,r)\subset \eta(x)$.
Since $\mu$ is ergodic, for $\mu$-a.e. $x$, there are infinite many
$n>0$ such that $f^n(x)\in S$.
Take $n_0=n_0(x)$ large enough, such that $f^{n_0}(x)\in S$
and $f^{-n_0}(\hat\xi(f^{n_0}(x)))\subset B^u(x, r)\subset \eta(x)$.
It follows that
$f^{-i}((f^{i-n_0}\hat\xi)(f^{i}(x)))\subset \eta(x)$ for any $i\ge n_0$.
Since $\hat\xi_{-k-i}=\vee_{j=0}^{k+i}f^j\hat\xi \ge f^{i-n_0}\hat\xi$,
$f^{-i}(\hat\xi_{-k-i}(f^{i}(x)))\subset \eta(x)$.
That is,
$\hat\xi_{-k-i}(f^{i}(x))\subset (f^{i}\eta)(f^i(x))$.
It implies that $(\hat\xi_{-k-i}\vee f^i\eta)(f^i(x))=(\hat\xi_{-k-i})(f^i(x))$
for all $i>N$.

By definition we can get directly
$I_\mu(\b|\hat\xi_{-k-i}\vee f^i\eta)(f^i(x))= I_\mu(\b|\hat\xi_{-k-i})(f^i(x))$.

Let $\phi_i=\big(I_\mu(\b|\hat\xi_{-k-i}\vee f^i\eta)
-I_\mu(\b|\hat\xi_{-k-i})\big)\circ f^i$.  The above fact gives
$\disp \lim_{i\to \infty} \phi_i(x)=0$ for almost every $x$.
By Fatou's lemma,
\[
\liminf_{i\to\infty}\int \phi_i d\mu
\ge \int \liminf_{i\to\infty}\phi_i d\mu
=0.
\]
It means
\[
\disp \liminf_{i\to\infty}H_\mu(\b|\hat\xi_{-k-i}\vee f^i\eta)
\ge \lim_{i\to\infty}H_\mu(\b|\hat\xi_{-k-i})=H_\mu(\b|\xi),
\]
where in the last step we use Lemma~\ref{Lcondzetan}(ii) for
$\zeta_i=\hat\xi_{-k-i}$ and $\zeta=\xi$.
Since $H_\mu(\b|\hat\xi_{-k-i}\vee f^i\eta)\le H_\mu(\b|\hat\xi_{-k-i})$
for any $i>0$, it follows
\[
\disp\limsup_{i\to \infty}H_\mu(\b|\hat\xi_{-k-i}\vee f^i\eta)
\le\lim_{i\to \infty} H_\mu(\b|\hat\xi_{-k-i}) =H_\mu(\b|\xi).
\]
Now we get $\disp\lim_{i\to \infty}H_\mu(\b|\hat\xi_{-k-i}\vee f^i\eta)
=H_\mu(\b|\xi)$.
\end{proof}

%%%%%%%%%%%%%%%%%%%%%%%%%%%%%%%%%%%%%%%%
%%%%%%%%%%%%%%%%%%%%%%%%%%%%%%%%%%%%%%%%
\subsection{Proof of Theorem A and its corollary}
%%%%%%%%%%%%%%%%%%%%%%%%%%%%%%%%%%%%%%%%

Note that $h_\mu(f, \xi)$ can be written as
$\lim_{n\to\infty}\frac{1}{n}H_\mu(\xi_0^{n-1}|f\xi)$,
and $h_\mu(f, \a|\eta)=\lim_{n\to\infty}\frac{1}{n}H_\mu(\a_0^{n-1}|f\eta)$.
To prove Theorem~A we need to know that the difference
$H_\mu(\xi_0^{n-1}|f\xi)- H_\mu(\a_0^{n-1}|f\eta)$ increases
at most subexponentially.
It is natural to compare both $H_\mu(\xi_0^{n-1}|f\xi)$ and
$H_\mu(\a_0^{n-1}|f\eta)$ with $H_\mu(\xi_0^{n-1}|\eta)$.
Since the size of elements of $\xi$ can be arbitrarily small,
it is unknown whether $H_\mu(\xi|\eta)$ is finite and
whether $H_\mu(\xi_0^{n-1}|\a_0^{n-1})$ increases at most
subexponentially.
So we cannot obtain the result as easy as the same way we use
in the proof in Lemma~\ref{Lcond}.

\begin{proposition}\label{PropA1}
Suppose $\mu$ is an ergodic measure.
Then for any $\a\in \P$, $\eta\in \P^u$ subordinate to unstable manifolds,
and $\xi\in \Q^u$,
\[
h_\mu(f,\a|\eta)\le h_\mu(f, \xi)
\]
\end{proposition}

\begin{proof}
Applying Lemma~\ref{Linduction}(i) with $\c=\eta$, $\b=\hat\xi_{-k}$,
and using the fact $\b_0^{n-1}=\hat\xi_{-k}^{n-1}$ and
$f^i\b_0^{i-1}=f\hat\xi_{-k-i+1}$
we have that for any $\eta\in \P^u$, $n>0$,
\[
\frac{1}{n}H_\mu(\hat\xi_{-k}^{n-1}|\eta)
=\frac{1}{n}H_\mu(\hat\xi_{-k}|\eta)
+\frac{1}{n}\sum_{i=1}^{n-1}H_\mu(\hat\xi_{-k}|f\hat\xi_{-k-i+1}\vee f^i\eta).
\]
Applying Lemma~\ref{Lconvergence} with $\b=\hat\xi_{-k}$,
we get that the terms in the summation converge to
$H_\mu(\hat\xi_{-k}|f\xi)$ as $i\to\infty$.
It is easy to see by the construction of $\hat\xi_{-k}$,
each element of $\eta$ intersects at most $2^k$ elements of $\hat\xi_{-k}$.
Hence $H_\mu(\hat\xi_{-k}|\eta)\le 2^k$ and
$\disp\frac{1}{n}H_\mu(\hat\xi_{-k}|\eta)\to 0$.   We get
\begin{equation}\label{fequiv1}
\lim_{n\to \infty}\frac{1}{n}H_\mu(\hat\xi_{-k}^{n-1}|\eta)
=H_\mu(\hat\xi_{-k}|f\xi)
\le H_\mu(\xi|f\xi).
\end{equation}

On the other hand, taking $\c=\hat\xi_{-k}^{n-1}$ in Lemma~\ref{Linduction}(ii)
we have
\begin{equation*}%\label{fequiv1a}
\begin{split}
H_\mu(\alpha_0^{n-1}|\hat\xi_{-k}^{n-1})
=&H_\mu(\alpha|\hat\xi_{-n-k+1})
   +\sum_{i=0}^{n-2}H_\mu(\alpha|\alpha_1^{n-1-i}\vee \hat\xi_{-k-i}^{n-i-1})\\
=&H_\mu(\alpha|\hat\xi_{-n-k+1})
   +\sum_{i=1}^{n-1}H_\mu(\alpha|\alpha_1^{i}\vee \hat\xi_{-k-n+1+i}^{i}) \\
\le &H_\mu(\alpha)
   +\sum_{i=1}^{n-1}H_\mu(\alpha|\alpha_1^{i}\vee \hat\xi_{-k}^{i}),
\end{split}
\end{equation*}
where we use the fact that $f^i\hat\xi_{-k}^{n-1}=\xi_{-k-i}^{n-i-1}$ and therefore
$f^{n-1}\hat\xi_{-k}^{n-1}=\hat\xi_{-n-k+1}$.
For any $\e>0$ we take $k>0$ as in Lemma~\ref{Lepsilonest}.
By the lemma we know that
$\limsup_{n\to \infty}H_\mu(\alpha|\alpha_1^{n}\vee \hat\xi_{-k}^{n-1})
\le \limsup_{n\to \infty}H_\mu(\alpha|\hat\xi_{-k}^{n-1}) \le \e$.
Since $H_\mu(\alpha)<\infty$, we get
\begin{equation}\label{fequiv2}
\limsup_{n\to \infty}\frac{1}{n}H_\mu(\alpha_0^{n-1}|\hat\xi_{-k}^{n-1}) \le \e.
\end{equation}

By Lemma~\ref{Lcond1}
\begin{equation}\label{fequiv4}
H_\mu(\alpha_0^{n-1}|\eta)
    \le H_\mu(\hat\xi_{-k}^{n-1}|\eta)+H_\mu(\alpha_0^{n-1}|\hat\xi_{-k}^{n-1}).
\end{equation}
By \eqref{fequiv2}, \eqref{fequiv4}, and then by \eqref{fequiv1},
we get
\begin{equation*}
\begin{split}
&h_\mu(f,\a|\eta)
=\limsup_{n\to\infty}\frac{1}{n}H_\mu(\alpha_0^{n-1}|\eta) \\
\le&\limsup_{n\to\infty}\frac{1}{n}H_\mu(\hat\xi_{-k}^{n-1}|\eta)+\e
=H_\mu(\xi|f\xi)+\e
=h_\mu(f, \xi)+\e.
\end{split}
\end{equation*}
Since $\e$ is arbitrary, we obtain the result of the proposition.
\end{proof}

\begin{proposition}\label{PropA2}
Suppose $\mu$ is an ergodic measure.
Then for any $\eta\in \P^u$ subordinate to unstable manifolds,
and $\xi\in \Q^u$,
\[
h_\mu(f, \xi)
\le \sup_{\a\in \P}h_\mu(f, \alpha|\eta).
\]
\end{proposition}

\begin{proof}
Recall that $\xi\in \Q^u$ is constructed after the statement of
Lemma~\ref{Lpartition}.
We take finite number of points $z^{(1)}, \cdots, z^{(K)}\in M$ and
real numbers $r^{(1)}, \cdots, r^{(K)}\le \tilde e$ such that
$\{S_u(z^{(1)}, r^{(1)}), \cdots, S_u(z^{(K)}, r^{(K)})\}$ form a cover of $M$,
where $S_u(z^{(j)}, r^{(j)})$ are the sets with the form given by \eqref{fdefS}
for each $1\le j\le K$.
Construct $\xi^{(j)}$ as the same way we described,
and denote $\tilde\xi=\xi^{(1)}\vee\cdots\vee \xi^{(K)}$.
Clearly $\tilde\xi$ is also a partition satisfying Lemma~\ref{Lpartition},
and every element of $\xi$ has diameter smaller than $\e_0$
if $\tilde e$ is small enough.
It is in fact proved in Lemma~3.1.2 in \cite{LY1} that for any two such
partitions $\xi'$ and $\xi''$,
$h_\mu(f, \xi'\vee \xi'')=h_\mu(f, \xi')$.
By induction we can show that $h_\mu(f, \tilde\xi)=h_\mu(f, \xi^{(j)})$
for any $1\le j\le K$.  So we only need to prove the result
for $\tilde \xi$.  For the sake of notational simplicity,
we will drop the tilde and write $\xi$ instead.

\begin{comment}
Note that $\eta\in \P^u$ is constructed from some finite partition in $\P$.
So each element of $\xi$ can intersect at most finitely number of
elements in $\eta$.  Hence $H_\mu(\eta|\xi)< \infty$ and therefore
\[
\lim_{n\to \infty}\frac{1}{n}H_\mu(\eta|\xi)=0.
\]
\end{comment}

Since $f^{-1}\xi$ is a measurable partition
of manifold $M$, there exists a sequence of partitions
$\alpha_n \in \P$ such that $\a_n\nearrow f^{-1}\xi$ as $n\to \infty$.
Hence, $\lim_{n \to \infty}H_\mu(\alpha_n|\xi)=H_\mu(f^{-1}\xi|\xi)$.
So we have
\[
\sup_{\alpha\in \P, \alpha < f^{-1}\xi}H_\mu(\alpha|\xi) = H_\mu(f^{-1}\xi|\xi).
\]
On the other hand, if $\a \in \P$ with $\alpha < f^{-1}\xi$,
then for any $i\geq 1$, $f^i\a_0^{i-1}< f^i(f^{-1}\xi)_0^{i-1}=\xi$.
By Lemma~\ref{Linduction}(i) with $\b=\a$, $\c=\xi$,
\begin{equation*}
H_\mu(\alpha_0^{n-1}|\eta)
=H_\mu(\alpha|\eta)+\sum_{i=1}^{n-1}H_\mu(\alpha|f^i\a_0^{i-1}\vee f^i\eta)
\ge H_\mu(\alpha|\eta)+\sum_{i=1}^{n-1}H_\mu(\alpha|\xi\vee f^i\eta)
\end{equation*}

By Lemma~\ref{Lconvergence} with $\b=\a$ we have
$\disp\lim_{i\to\infty}H_\mu(\alpha|\xi\vee f^i\eta)=H_\mu(\alpha|\xi)$.
Hence
\begin{equation*}%\label{fPropA21}
    \limsup_{n\to\infty}\frac{1}{n} H_\mu(\alpha_0^{n-1}|\eta)
\ge \liminf_{n\to\infty}\frac{1}{n} H_\mu(\alpha_0^{n-1}|\eta)
\ge H_\mu(\alpha|\xi).
\end{equation*}
So we get
\begin{equation}\label{fPropA2}
\begin{split}
\sup_{\alpha\in \P}h_\mu(f, \a|\eta)
\ge &\sup_{\alpha\in \P, \alpha < f^{-1}\xi}h_\mu(f, \a|\eta)
=\sup_{\alpha\in \P, \alpha < f^{-1}\xi}\limsup_{n\to\infty}
    \frac{1}{n} H_\mu(\alpha_0^{n-1}|\eta)                    \\
\ge&\sup_{\alpha\in \P, \alpha < f^{-1}\xi}\liminf_{n\to\infty}
    \frac{1}{n} H_\mu(\alpha_0^{n-1}|\eta)
\ge \sup_{\alpha\in \P, \alpha < f^{-1}\xi}H_\mu(\alpha|\xi)   \\
=&H_\mu(f^{-1}\xi|\xi)=h_\mu(f, \xi).
\end{split}
\end{equation}
This is what we need.
\end{proof}

\begin{proof}[Proof of Theorem A]
Proposition~\ref{PropA1} and Proposition~\ref{PropA2} gives that
for any $\a\in \P$, $\eta\in \P^u$ subordinate to unstable manifolds,
and $\xi\in \Q^u$,
\[
h_\mu(f, \a|\eta)
\le h_\mu(f, \xi)
\le \sup_{\b\in \P}h_\mu(f, \b|\eta).
\]
By Lemma~\ref{Lcond}, $\sup_{\b\in \P}h_\mu(f, \b|\eta)=h_\mu(f, \a|\eta)$.
So the result follows.

By Lemma~\ref{Lcond}, $h_\mu(f, \a|\eta)$ is independent of choice of
$\eta$ as long as $\eta\in \P^u$.  So the result is true for
any $\eta\in \P^u$, not necessary subordinate to unstable manifolds.
\end{proof}

\begin{proof}[Proof of Corollary A.1]
Since
$H_\mu(\a_0^{n-1}|\eta)\le H_\mu(\a_0^{n-1})$ for any $\a\in \P$
and $\eta\in \P^u$, by definition we get $h_\mu^u(f)\le h_\mu(f)$.

If $f$ is $C^{r}$ with $r>1$, then Ledrappier-Young's formula
can be applied, that is,
\[
h_\mu(f)=\sum_{i\le \tilde u}\lambda_i \gamma_i,
\]
where $\lambda_1> \cdots > \lambda_{\tilde u}>0$ are the positive Lyapunov
exponents, $0\le \gamma_i\le \dim E_i$, and $E_i$ are the subspaces
whose nonzero vectors have Lyapunov exponents $\lambda_i$.
If there are $u$ distinct Lyapunov exponents on the unstable subspace,
then $u\le \tilde u$.  Ledrappier-Young's formula also gives
$H_{\mu}(\xi|f\xi)=\sum_{i\le u}\lambda_i \gamma_i$, where $\xi\in \Q^u$.
Since by Theorem~A, $h_\mu^u(f)=H_{\mu}(\xi|f\xi)$,
we get the inequalities.

If there is no positive Lyapunov exponent in the center direction, then
$\tilde u=u$, hence we can take $\xi_{\tilde u}=\xi_{u}$ to get
$h_\mu^u(f)=h_\mu(f, \xi_{u})=h_\mu(f, \xi_{\tilde u})=h_\mu(f)$.
\end{proof}

\begin{proof}[Proof of Corollary A.2]
The equality $h_\mu^u(f)=h_\mu^u(f, \a|\eta)$ for any $\a\in \P$
and $\eta\in \P^u$ is implied in Lemma~\ref{Lcond},
as well as in the two equalities given in Theorem A.
So we only need to prove that the limit
$\lim_{n\to \infty}\frac{1}{n}H_\mu(\alpha_0^{n-1}|\eta)$ exists.

First, by Theorem A, all ``$\ge$'' in \eqref{fPropA2} becomes ``$=$''.
Also, by Lemma~\ref{Lcond}, all the supremum in \eqref{fPropA2}
can be dropped.  So \eqref{fPropA2} becomes
\begin{equation*}
h_\mu(f, \a|\eta)
= \limsup_{n\to\infty}    \frac{1}{n} H(\alpha_0^{n-1}|\eta)
=\liminf_{n\to\infty}     \frac{1}{n} H(\alpha_0^{n-1}|\eta)
= h_\mu(f, \xi).
\end{equation*}
We get existence of $\lim_{n\to\infty} \frac{1}{n} H(\alpha_0^{n-1}|\eta)$.
\end{proof}

%%%%%%%%%%%%%%%%%%%%%%%%%%%%%%%%%%%%%%%%%%%%%%%%
%%%%%%%%%%%%%%%%%%%%%%%%%%%%%%%%%%%%%%%%%%%%%%%%
\subsection{Further Properties}
%%%%%%%%%%%%%%%%%%%%%%%%%%%%%%%%%%%%%%%%%%%%%%%%

In this subsection we show that the unstable metric entropy is
affine and upper semi-continuous with respect to measures.

Recall that $\mathcal{M}_f(M)$ and $\mathcal{M}^e_f(M)$ denote the set
of all $f$-invariant and ergodic probability measures on $M$ respectively.
Let $\mathcal{M}(M)$ denote the set of all probability measures on $M$.

Note that any partition $\c$ generates a sub-$\sigma$-algebra
${\mathcal B}(\c)$, that is, ${\mathcal B}(\c)$ is the smallest
sub-$\sigma$-algebra that contains the elements in the partition $\c$.
Clearly, if $\{\c_n\}$ is a sequence of increasing measurable partitions,
then  $\{{\mathcal B}(\c_n)\}$ is a sequence of increasing
sub-$\sigma$-algebras.

\begin{proposition}\label{Paffine}
For any $\a\in \P$ and $\eta\in \P^u$, the map
$\mu \mapsto H_\mu(\alpha|\eta)$
from $\mathcal{M}(M)$ to ${\mathbb R}^+\cup\{0\}$ is concave.
%{\color{blue}Why do we need it?  It is weaker than affineness.}

Furthermore, the map $\mu \mapsto h_\mu^u(f)$ from $\mathcal{M}_f(M)$ to
${\mathbb R}^+\cup\{0\}$ is affine.
\end{proposition}

\begin{proof}
For $\mu=a\mu_1+(1-a)\mu_2$ where
$\mu_1, \mu_2 \in \mathcal{M}(M)$ and $0<a<1$,
and for any $\alpha, \beta \in \P$, it is well known that (cf. Lemma 3.3 in \cite{HYZ})
\begin{equation*}\label{faffine}
0\leq H_\mu(\alpha|\beta)- aH_{\mu_1}(\alpha|\beta)-(1-a)H_{\mu_2}(\alpha|\beta)\leq \phi(a)+\phi(1-a)
\end{equation*}
where $\phi(x)=-x\log x$. For $\eta\in P^u$, we can find a sequence of partitions $\beta_n \in \P$ such that $\beta_1<\beta_2< \cdots$.
Using Lemma~\ref{Lcondzetan} with $\zeta_n=\beta_n$ and
$\zeta=\eta$, we have
\begin{equation*}
\begin{aligned}
H_\mu(\alpha|\eta)&=\lim_{n\to \infty}H_{\mu}(\alpha|\beta_n)
\geq \lim_{n\to \infty} \left(aH_{\mu_1}(\alpha|\beta_n)
  +(1-a)H_{\mu_2}(\alpha|\beta_n)\right)\\
&=aH_{\mu_1}(\alpha|\eta)+(1-a)H_{\mu_2}(\alpha|\eta).
\end{aligned}
\end{equation*}
The first part of the proposition follows. Similarly, we have
\begin{equation*}
\begin{aligned}
H_\mu(\alpha|\eta)\leq aH_{\mu_1}(\alpha|\eta)+(1-a)H_{\mu_2}(\alpha|\eta)+\phi(a)+\phi(1-a).
\end{aligned}
\end{equation*}
Hence
\begin{equation*}
\begin{aligned}
&aH_{\mu_1}(\alpha_0^{n-1}|\eta)+(1-a)H_{\mu_2}(\alpha_0^{n-1}|\eta)\leq H_\mu(\alpha_0^{n-1}|\eta)\\
\leq &aH_{\mu_1}(\alpha_0^{n-1}|\eta)+(1-a)H_{\mu_2}(\alpha_0^{n-1}|\eta)+\phi(a)+\phi(1-a).
\end{aligned}
\end{equation*}
Dividing by $n$ and taking limit, we have $h_\mu(f,\alpha|\eta)=ah_{\mu_1}(f,\alpha|\eta)+(1-a)h_{\mu_2}(f,\alpha|\eta)$. Then the second part of the proposition follows by Corollary~A.2.
\end{proof}

Recall that for each partition $\a\in \P$, the partition $\zeta$
given by $\zeta(x)=\a(x)\cap W^u_\loc(x)$ for any $x\in M$
is an element in $\P^u$.  Denote such $\zeta$ by $\a^u$.
Conversely, for each partition $\eta\in \P^u$, there is a partition
$\b\in \P$ such that $\eta(x)=\b(x)\cap W^u_\loc(x)$ for any $x\in M$.
Denote such $\b$ by $\eta^{\tilde{u}}$.

\begin{proposition}\label{Pusc}
(a) Let $\nu\in \mathcal{M}(M)$.
For any $\a\in \P$ and $\eta\in \P^u$ with $\mu(\partial \a)=0$
and $\mu(\partial \eta^{\tilde{u}})=0$,
the map $\mu \mapsto H_\mu(\alpha|\eta)$ from $\mathcal{M}(M)$ to
${\mathbb R}^+\cup\{0\}$ is upper semi-continuous at $\mu$,
i.e.
$$
\limsup_{\nu\to \mu}H_{\nu}(\alpha|\eta)\leq H_\mu(\alpha|\eta).
$$

(b) The unstable entropy map $\mu \mapsto h_\mu^u(f)$ from
$\mathcal{M}_f(M)$ to ${\mathbb R}^+\cup\{0\}$ is
upper semi-continuous at $\mu$.
i.e.
$$
\limsup_{\nu\to \mu}h^u_{\nu}(f)\leq h_\mu^u(f).
$$
\end{proposition}

\begin{proof}
(a) Since $\mu(\partial \eta^{\tilde{u}})=0$,
we can take a sequence of partitions
$\{\beta_n\} \subset \mathcal{P}$ such that $\beta_1<\beta_2< \cdots$
and $\mathcal{B}(\beta_n)\nearrow\mathcal{B}(\eta)$,
and moreover, $\mu(\partial \beta_n)=0$ for $n=1,2, \cdots$.

Since $\mu(\partial \alpha)=0=\mu(\partial \beta_n)$, and
for any invariant measure $\nu$,
\[
H_\nu(\alpha|\beta_n)
=-\sum_{A_i\in \alpha,B_j\in \beta_n}\nu(A_i \cap B_j)
 \log \frac{\nu(A_i \cap B_j)}{\nu(B_j)},
\]
we have $\lim_{\nu\to \mu}H_{\nu}(\alpha|\beta_n)= H_\mu(\alpha|\beta_n)$
for any $n\in \mathbb{N}$.
By martingale convergence theorem,
$H_\nu(\alpha|\eta)=\lim_{n\to \infty}H_\nu(\alpha|\beta_n)$.
So for any $\epsilon>0$, there exists $N\in \mathbb{N}$ such that
$H_\mu(\alpha|\beta_N)\leq H_\mu(\alpha|\eta)+\epsilon.$
One has
$$
\limsup_{\nu\to \mu}H_{\nu}(\alpha|\eta)
\leq \limsup_{\nu\to \mu}H_{\nu}(\alpha|\beta_N)
=H_\mu(\alpha|\beta_N)
\leq H_\mu(\alpha|\eta)+\epsilon.
$$
Since $\epsilon>0$ is arbitrary, we get the inequality.

(b) To get the upper semi-continuity for unstable entropy map
$\mu \mapsto h_\mu^u(f)$, we take $\a\in \P$ and $\eta\in \P^u$
with $\mu(\partial \a)=0$ and $\mu(\partial \eta^{\tilde{u}})=0$.

By Lemma~\ref{Lcond1}, for any $f$-invariant measure $\nu$, we have
\begin{equation}\label{fsubadd1}
\begin{aligned}
H_\nu(\alpha_0^{m+n-1}|\eta)
=&H_\nu(\alpha_0^{n-1}|\eta)+H_\nu(f^{-n}\alpha_0^{m-1}|\alpha_0^{n-1}\vee\eta)\\
=&H_\nu(\alpha_0^{n-1}|\eta)+H_\nu(\alpha_0^{m-1}|\alpha_{-n}^{-1}\vee f^n\eta)\\
\le &H_\nu(\alpha_0^{n-1}|\eta)+H_\nu(\alpha_0^{m-1}|\eta)
   +H_\nu(\eta|\alpha_{-n}^{-1}\vee f^n\eta).
\end{aligned}
\end{equation}

Note that for any $\zeta\in \P^u$, $x\in M$,
$\eta^{\tilde{u}}(x)\cap \zeta(x) = \eta(x)\cap \zeta(x)$. The definition
of conditional entropy gives
\[
H_\nu(\eta|\alpha_{-n}^{-1}\vee f^n\eta)
=H_\nu(\eta^{\tilde{u}}|\alpha_{-n}^{-1}\vee f^n\eta)
\le H_\nu(\eta^{\tilde{u}}).
\]
So by \eqref{fsubadd1},
\[
    H_\nu(\alpha_0^{m+n-1}|\eta)
\le H_\nu(\alpha_0^{n-1}|\eta)+H_\nu(\alpha_0^{m-1}|\eta)
   +H_\nu(\eta^{\tilde{u}}).
\]
That is, $\{H_\nu(\alpha_0^{n-1}|\eta)+H_\nu(\eta^{\tilde{u}})\}$ is
a subadditive sequence.  Hence we have
\begin{equation}\label{fsubadd2}
\lim_{n\to \infty}\frac{1}{n}
          \big(H_\nu(\alpha_0^{n-1}|\eta)+H_\nu(\eta^{\tilde{u}})\big)
=\inf_{n\in \mathbb{N}} \frac{H_\nu(\alpha_0^{n-1}|\eta)+H_\nu(\eta^{\tilde{u}})}{n}.
\end{equation}

Let $\e>0$ be arbitrary.  By Corollary~A.2, we can take $N\in \mathbb{N}$
large enough such that
\[
\disp\frac{H_\mu(\alpha_0^{N-1}|\eta)+H_\mu(\eta^{\tilde{u}})}{N}
\leq  h_\mu(f,\alpha|\eta)+\e= h_\mu^u(f)+\e.
\]
Since $\mu(\partial \a)=0$, $\mu(\partial \a_0^{n-1})=0$ for any $n\ge 1$
by invariance of measure $\mu$.  So we can use Corollary~A.2 and
\eqref{fsubadd2}, and then apply the conclusion
in part~(a) with $\a$ replaced by $\a_0^{n-1}$ to get
\begin{equation*}
\begin{aligned}
&\limsup_{\nu\to \mu}h_\nu^u(f)
=\limsup_{\nu\to \mu}h_\nu^u(f,\alpha|\eta)
=\limsup_{\nu\to \mu}\inf_{n\in \mathbb{N}}
  \frac{H_\nu(\alpha_0^{n-1}|\eta)+H_\nu(\eta^{\tilde{u}})}{n} \\
\leq &\limsup_{\nu\to \mu}\frac{H_\nu(\alpha_0^{N-1}|\eta)+H_\nu(\eta^{\tilde{u}})}{N}
\leq \frac{H_\mu(\alpha_0^{N-1}|\eta)+H_\mu(\eta^{\tilde{u}})}{N}
\leq h_\mu^u(f)+\e.
\end{aligned}
\end{equation*}
Since $\e>0$ is arbitrary, we get the result.
\end{proof}

%%%%%%%%%%%%%%%%%%%%%%%%%%%%%%%%%%%%%%%%%%%%%%%%
%%%%%%%%%%%%%%%%%%%%%%%%%%%%%%%%%%%%%%%%%%%%%%%%
\section{Shannon-McMillan-Breiman Theorem}\label{SSMBThm}
\setcounter{equation}{0}
%%%%%%%%%%%%%%%%%%%%%%%%%%%%%%%%%%%%%%%%%%%%%%%%
In this section, we prove Theorem B,
a version of Shannon-McMillan-Breiman theorem for unstable metric entropy.
Throughout the section we assume that $\mu$ is an ergodic measure of $f$
since we have such an assumption in the theorem.

\begin{comment}
Recall Lemma~{Linduction}(ii), where $I_\mu(\alpha_0^{n-1}|\eta)(x)$
can be expressed as a sum of information function of the form
$I_\mu(\a|\a_1^{n-1}\vee f^i\eta)(f^i(x))$.
So the proof consists two main ingredients a generalized ergodic theorem and
the martingale convergence theorem, even though $I_\mu(\a|\a_1^{n-1}\vee f^i\eta)$
is not really a martingale.
\end{comment}

%%%%%%%%%%%%%%%%%%%%%%%%%%%%%%%%%%%%%%%%%
%%%%%%%%%%%%%%%%%%%%%%%%%%%%%%%%%%%%%%%%%
\subsection{Proof of Theorem B: Lower Limits}
%%%%%%%%%%%%%%%%%%%%%%%%%%%%%%%%%%%%%%%%%i

Let $d^u$ denote the metric induced by the Riemmanian structure
on unstable manifolds.
Let $B^u(y,r)$ denote the open ball centered at $y$ with radius $r>0$
in the unstable manifold $W^u(y)$ with respect to $d^u$.

Then let $d^u_{n}(x,y)=\max _{0 \leq j \leq n-1}d^u(f^j(x),f^j(y))$,
and $B^u_{n}(x,\epsilon)$ be the open ball centered at $x$ of radius $\e$
with respect to the metric $d^u_n$, i.e. an
\emph{$(n,\epsilon)$ Bowen ball in $W^u(x)$ about $x$}.
%It is easy to see that $d^u_{n}(x,y)=d^u(f^{n-1}(x),f^{n-1}(y))$.

\begin{comment}
Theorem B is concern the increasing rates of $I_\mu(\a_0^{n-1}|\eta)(x)$.
We will replace $\a$ by $\bar\a$, where $\bar\a(x):=\a(x)\cap W^u_\loc(x)$.
(So in fact $\bar\a\in \P^u$.)  We can do so since
$I_\mu(\bar\a_0^{n-1}|eta)(x)=I_\mu(\bar\a_0^{n-1}|eta)(x)$.
By abuse of notation we still write $\a$ instead and state $\a\in \P$.
for the sake of simplicity.
\end{comment}

Recall for any $\xi\in \Q^u$, the entropy of $\xi$ is given by
$h_\mu(f, \xi)=H_\mu(f^{-1}\xi|\xi)$.

\begin{lemma}\label{LSMB1.1}
For any $\xi\in \Q^u$, $\e>0$,
\begin{equation*}
h_\mu(f, \xi)%=H_\mu(f^{-1}\xi|\xi)
=\lim_{n\to \infty}-\frac{1}{n}\log\mu_x^{\xi}(B^u_{n}(x,\epsilon)), \quad
\mu\ae x.
\end{equation*}
\end{lemma}

\begin{proof}
Denote
\begin{equation}\label{flocalentropy1}
\begin{split}
&\underline{h}_u(f, x, \epsilon,\xi)
=\liminf_{n\to \infty}-\frac{1}{n}\log\mu_x^{\xi}(B^u_{n}(x,\epsilon)),
  \\
&\overline{h}_u(f, x, \epsilon, \xi)
=\limsup_{n\to \infty}-\frac{1}{n}\log\mu_x^{\xi}(B^u_{n}(x,\epsilon)).
\end{split}
\end{equation}
It is proved in (9.2) and (9.3) in \cite{LY2} that
\begin{equation}\label{flocalentropy2}
\lim_{\epsilon\to 0}\underline{h}_u(f, x, \epsilon, \xi)
=\lim_{\epsilon\to 0}\overline{h}_u(f, x, \epsilon, \xi),
\end{equation}
and hence the \mbox{common} value can be denoted as $h_u(f,x,\xi)$.
When $\mu$ is ergodic, $x\mapsto h_u(f,x,\xi)$ is constant almost
everywhere. This constant coincides with $h_\mu(f, \xi)$.%=H_\mu(\xi|f\xi)$.

Now we show that the upper and lower limits in \eqref{flocalentropy1}
are limit.

Since $f$ is uniformly expanding restricted to the unstable manifolds,
for any $0<\d<\e$, there exists $k>0$ such that
$B^u_k(x, \e)\subset B^u(x, \d)$ and therefore
$B^u_{n+k}(x, \e)\subset B^u_n(x, \d)\subset B^u_n(x, \e)$ for any
$n>0$ and $x\in M$.  It implies
\begin{equation*}
\begin{split}
\liminf_{n\to \infty}-\frac{1}{n}\log\mu_x^{\xi}(B^u_{n}(x,\d))
=&\liminf_{n\to \infty}-\frac{1}{n}\log\mu_x^{\xi}(B^u_{n}(x,\epsilon)),  \\
\limsup_{n\to \infty}-\frac{1}{n}\log\mu_x^{\xi}(B^u_{n}(x,\d))
=&\limsup_{n\to \infty}-\frac{1}{n}\log\mu_x^{\xi}(B^u_{n}(x,\epsilon)).
\end{split}
\end{equation*}
It means that both $\overline{h}_i(f, x, \epsilon, \xi)$ and
$\underline{h}_i(f, x, \epsilon, \xi)$ given in \eqref{flocalentropy1}
are independent of $\e$.  So \eqref{flocalentropy2}becomes $\overline{h}_u(f, x, \e, \xi)=\underline{h}_u(f, x, \e, \xi)$.
We get that for $\mu$-a.e. $x$,
\begin{equation*}
H_\mu(f^{-1}\xi|\xi)=h_u(f,x,\xi)
=\lim_{n\to \infty}-\frac{1}{n}\log\mu_x^{\xi}(B^u_{n}(x,\epsilon)). \qedhere
\end{equation*}
\end{proof}

\begin{corollary}\label{Cpointwise}
For any $\eta\in \P^u$ subordinate to unstable manifolds and any $\e>0$,
\begin{equation*}
h_\mu(f|\eta)%=H_\mu(f^{-1}\xi|\xi)
=\lim_{n\to \infty}-\frac{1}{n}\log\mu_x^{\eta}(B^u_{n}(x,\epsilon)) \quad
\mu\ae x.
\end{equation*}
\end{corollary}

\begin{proof}
Take $\xi\in \Q^u$, $\e>0$.  Let $x\in M$ be a generic point.
Since $\eta$ is subordinate to unstable manifolds, there exists $N>0$
such that for any $n>N$, $B^u_{n}(x,\epsilon)\subset \eta(x)$.

Suppose $\eta(x)\subset \xi(x)$, then
\[
\mu_x^\xi(B^u_{n}(x,\epsilon))
=\mu_x^{\eta}(B^u_{n}(x,\epsilon))\mu_x^{\xi}(\eta(x)).
\]
Since for $\mu$-a.e. $x$, $\mu_x^{\xi}(\eta(x))$ is finite,
$\lim_{n\to \infty}-\frac{1}{n}\log\mu_x^{\xi}(\eta(x))=0$.
So by Theorem A and Lemma~\ref{LSMB1.1} we get
\[
h_\mu(f|\eta)=h_\mu(f, \xi)
=\lim_{n\to \infty}-\frac{1}{n}\log\mu_x^{\eta}(B^u_{n}(x,\epsilon)).
\]

Suppose $\eta(x)\not\subseteq \xi(x)$.
Then for $\mu$ almost every $x$, we can take $k=k(x)>0$ such that
$f^{-k}(\eta(x))\subset \xi(f^{-k}(x))$.
In fact, let $\Omega_r:=\{y\in M: B^u(y,r)\subset \xi(y)\}$.
Since $\xi$ is subordinate to $W^u$, $\mu(\cup_{r>0}\Omega_r)=1$.
So there exists $r>0$ such that $\mu(\Omega_r)>0$.
For $\mu$-a.e. $x$, there exist infinitely many $k_i=k_i(x)$
such that $f^{-k_i}x\in \Omega_r$.
Let $k_i>0$ be large enough such that
$f^{-k_i}(\eta(x))\subset B^u(f^{-k_i}(x), r)\subset \xi(f^{-k_i}(x))$. Let $k=k_i$.

Now we have
\begin{equation*}
\begin{split}
h_\mu(f|f^{-k}\eta)%=H_\mu(f^{-1}\xi|\xi)
=&\lim_{n\to \infty}-\frac{1}{n}\log\mu_{f^{-k}x}^{f^{-k}\eta}
     (B^u_{n}(f^{-k}(x),\epsilon)) \\
=&\lim_{n\to \infty}-\frac{1}{n}\log\mu_{f^{-k}x}^{f^{-k}\eta}
     (B^u_{n+k}(f^{-k}(x),\epsilon)),
\end{split}
\end{equation*}
Since $f^{-k}\eta\in \P^u$, we have $h_\mu(f|f^{-k}\eta)=h_\mu(f|\eta)$
by Lemma~\ref{Lcond}.  Also, since $\mu$ is an invariant measure
and $f^k(B^u_{n+k}(f^{-k}(x),\epsilon))=B^u_{n}(x,\epsilon)$, it follows that
$\mu_{f^{-k}x}^{f^{-k}\eta}(B^u_{n+k}(f^{-k}(x),\epsilon))
=\mu_x^{\eta}(B^u_{n}(x,\epsilon))$.
So the result of the corollary follows.
\end{proof}

\begin{lemma}\label{LSMB1.2}
Let $\a\in \P$, $\eta\in \P^u$.  Then for any $\xi\in \Q^u$,

\begin{equation*}
h_\mu(f, \a|\eta)
\leq \liminf_{n\to \infty}\frac{1}{n}I_\mu(\alpha_0^{n-1}|\xi)(x) \quad
\mu\ae x.
\end{equation*}
\end{lemma}

\begin{proof}
Let $\e>0$. Take $k>0$ such that $\diam \a^k_0\vee \xi\le \e$.
Hence, for any $n>0$,
$(\a_0^{k+n-1}\vee\xi)(x)=\vee_{i=0}^{n-1} (f^{-i}\a^k_0\vee \xi)(x)
\subset B^u_n(x,\e)$.  By Theorem~A and Lemma~\ref{LSMB1.1},
\begin{equation*}
\begin{split}
h_\mu(f, \a|\eta)=&h_\mu(f, \xi)
=\lim_{n\to \infty}-\frac{1}{n}\log\mu_x^\xi(B^u_{n}(x,\epsilon))
 \leq \liminf_{n\to \infty} -\frac{1}{n}\log\mu_x^\xi(\alpha_0^{k+n-1}(x)) \\
=&\liminf_{n\to \infty}-\frac{1}{n}\log\mu_x^\xi(\alpha_0^{n-1}(x))
= \liminf_{n\to \infty}\frac{1}{n}I_\mu(\alpha_0^{n-1}|\xi)(x).
\end{split}
\end{equation*}
for $\mu$-a.e. $x$.
\end{proof}

Next, we need to pass the given partition from $\xi$ to $\eta$ to obtain
the estimates we want.

\begin{lemma}\label{LSMB1.3}
Let $\a\in \P, \eta\in \P^u$ and $\xi\in \Q^u$.  Then for $\mu$-a.e. $x$,
\begin{equation*}
\begin{split}
\liminf_{n\to \infty}\frac{1}{n}I_\mu(\a_0^{n-1}|\xi)(x)
= &\liminf_{n\to \infty}\frac{1}{n}I_\mu(\a_0^{n-1}|\eta)(x),  \\
\limsup_{n\to \infty}\frac{1}{n}I_\mu(\a_0^{n-1}|\xi)(x)
= &\limsup_{n\to \infty}\frac{1}{n}I_\mu(\a_0^{n-1}|\eta)(x).
\end{split}
\end{equation*}
\end{lemma}

\begin{proof}
%Recall that $\hat\xi_{-k}=\hat\xi_{-k}^{0}$ is defined in \eqref{fdefxik}.
By Lemma~\ref{Lcond1}, for $\mu$-a.e. $x$,
\begin{equation}\label{fcondinfor1}
\begin{split}
 &I_\mu(\alpha_0^{n-1}|\xi)
        +I_\mu(\eta|\alpha_0^{n-1}\vee \xi)
=I_\mu(\alpha_0^{n-1}\vee \eta|\xi)
=I_\mu(\alpha_0^{n-1}|\xi\vee \eta)
        +I_\mu(\eta|\xi),    \\
 &I_\mu(\alpha_0^{n-1}|\eta)
        +I_\mu(\xi|\alpha_0^{n-1}\vee \eta)
=I_\mu(\alpha_0^{n-1}\vee \xi|\eta)
=I_\mu(\alpha_0^{n-1}|\xi\vee \eta)
        +I_\mu(\xi|\eta).
\end{split}
\end{equation}
Since $\diam (\a_0^{n-1}\vee\xi)(x)\to 0$ and
$\diam (\a_0^{n-1}\vee\eta)(x)\to 0$ as $n\to \infty$ for any $x$,
\[
\disp\lim_{n\to \infty}\frac{1}{n}I_\mu(\eta|\a_0^{n-1}\vee\xi)(x)
=0
=\disp\lim_{n\to \infty}\frac{1}{n}I_\mu(\xi|\a_0^{n-1}\vee\eta)(x)
\quad \mu-\text{a.e.}\ x,
\]
Also, since $I_\mu(\xi|\eta)$ and $I_\mu(\eta|\xi)$ are finite $\mu$
almost everywhere,
\[
\disp\lim_{n\to \infty}\frac{1}{n}I_\mu(\xi|\eta)(x)
=0
=\disp\lim_{n\to \infty}\frac{1}{n}I_\mu(\eta|\xi)(x)
\quad \mu-\text{a.e.}\ x.
\]

By \eqref{fcondinfor1}, for $\mu$ almost every $x$,
\begin{equation*}
I_\mu(\alpha_0^{n-1}|\xi)
        +I_\mu(\eta|\alpha_0^{n-1}\vee \xi)+I_\mu(\xi|\eta)
=I_\mu(\alpha_0^{n-1}|\eta)
        +I_\mu(\xi|\alpha_0^{n-1}\vee \eta) +I_\mu(\eta|\xi),
\end{equation*}
Dividing by $n$, and taking $\liminf$ and $\limsup$,
we can get the equalities.
\end{proof}

\begin{proof}[Proof of Theorem B: the Lower Limits]
By Lemma~\ref{LSMB1.2} and \ref{LSMB1.3} we get directly
\begin{equation*}
h_\mu(f, \a|\eta)
\le \liminf_{n\to \infty}\frac{1}{n}I_\mu(\alpha_0^{n-1}|\eta)(x). \qedhere
\end{equation*}
\end{proof}

%%%%%%%%%%%%%%%%%%%%%%%%%%%%%%%%%%%%%%%%%
%%%%%%%%%%%%%%%%%%%%%%%%%%%%%%%%%%%%%%%%%
\subsection{A generalized ergodic theorem}
%%%%%%%%%%%%%%%%%%%%%%%%%%%%%%%%%%%%%%%%%

The next results can be viewed as generalizations of Birkhoff ergodic theorem.
The results and methods of proof can be seen in references for some
particular sequence of functions
(e.g. \cite{Pet}, proof of Theorem 2.3 on p.261).
We state it in a more general setting.

\begin{proposition}\label{Lergodic}
Let $T: X\to X$ be a transformation preserving an ergodic measure $\mu$.
Suppose $\{\phi_n\}$ is a sequence of functions on $X$
satisfying the following:
\begin{enumerate}
\item[(i)]  $\disp \lim_{n\to\infty} \phi_n(x)=\phi_0(x)$ \ $\mu$-a.e. $x$,
for some function $\phi_0\in L^1(\mu)$;
\item[(ii)] $\phi^*:=\sup_{n}|\phi_n| \in L^1(\mu)$.
\end{enumerate}
Then for $\mu$-a.e. $x\in X$,
\[
\lim_{n\to\infty}\frac{1}{n}\sum_{i=0}^{n-1}\phi_{n-i}(T^i(x))
=\lim_{n\to\infty}\frac{1}{n}\sum_{i=0}^{n-1}\phi_0(T^i(x))
=\int \phi_0 d\mu,
\]
and
\[
\lim_{n\to\infty}\frac{1}{n}\sum_{i=0}^{n-1}\phi_{i}(T^i(x))
=\lim_{n\to\infty}\frac{1}{n}\sum_{i=0}^{n-1}\phi_0(T^i(x))
=\int \phi_0 d\mu.
\]
\end{proposition}

\begin{proof}
We may assume $\phi_0=0$, otherwise we can replace $\phi_n$ by
$\phi_n-\phi_0$ and take $\phi^*=\sup_{n}|\phi_n-\phi_0|$.

Let $\e>0$ be given. Denote $\phi_N^*=\sup_{n\ge N}\phi_n$, then
$\phi_N^*\le \phi^*$ and $\phi_N^*\to 0$ \
$\mu$-a.e. as $N\to \infty$.
 By Lebesgue's dominated convergence theorem, we have
\[
\lim_{N\to \infty} \int \phi_N^* d\mu
=\int \lim_{N\to \infty} \phi_N^* d\mu
=0.
\]
So we can take $N_0>0$ such that for any $N \ge N_0$, $\int \phi_N^* d\mu < \e$.

By Birkhoff ergodic theorem there exists $N_1=N_1(x, N)>N$ such that for any
$n>N_1$,
\[
\frac{1}{n} \sum_{i=0}^{n-1} \phi_{N}^*(T^ix)
\le \Big|\frac{1}{n}\sum_{i=0}^{n-1} \phi_{N}^*(T^ix)-\int \phi_N^* d\mu \Big|
+\int \phi_N^* d\mu  < 2\e.
\]
Hence
\[
\frac{1}{n}\Big| \sum_{i=0}^{n-N-1} \phi_{n-i}(T^ix)\Big|
\le \frac{1}{n}\sum_{i=0}^{n-N-1} \phi_{N}^*(T^ix)
\le \frac{1}{n}\sum_{i=0}^{n-1} \phi_{N}^*(T^ix)<2\e.
\]

On the other hand, Birkhoff ergodic theorem implies
$\dfrac{1}{n} \phi^*(T^ny)\to 0$ $\mu$-a.e. $y$.
Hence, there exists $N_2=N_2(x, N)>N$ such that for any $n>N_2$,
\[
\frac{1}{n}\Big|\sum_{i=n-N}^{n-1} \phi_{n-i}(T^ix)\Big|
\le \frac{1}{n}\sum_{i=1}^{N} |\phi_{i}(T^{n-i}x)|
\le \frac{1}{n}\sum_{i=1}^{N} \phi^*(T^{n-i}x)
< \e.
\]
So we get that for any $n > \max\{N_0, N_1, N_2\}$,
$\disp \frac{1}{n}\Big|\sum_{i=0}^{n-1} \phi_{n-i}(T^ix)\Big|< 3\e$.
Hence we obtain the first equality.

The second formula can be proved in a similar way.
\end{proof}

\begin{remark}
If $\phi_n = \phi$ for all $n$ in the above proposition,
then it is Birkhoff ergodic theorem.
\end{remark}

%%%%%%%%%%%%%%%%%%%%%%%%%%%%%%%%%%%%%%%%%
%%%%%%%%%%%%%%%%%%%%%%%%%%%%%%%%%%%%%%%%%
\subsection{Proof of Theorem B: Upper Limits}
%%%%%%%%%%%%%%%%%%%%%%%%%%%%%%%%%%%%%%%%%
First we show the result with $\a$ replaced by an increasing
partition $\xi$, which is easy to get.
Then we use Lemma~\ref{Lmain} below to pass the result
from $H_\mu(\xi_0^{n-1}|\eta)$ to $H_\mu(\a_0^{n-1}|\eta)$.

\begin{lemma}\label{LSMB1}
For any $\eta\in \P^u$ and $\xi\in \Q^u$,
\begin{equation*}
\lim_{n\to \infty}\frac{1}{n}I_\mu(f^{-n}\xi|\eta)(x)
=\lim_{n\to \infty}\frac{1}{n}I_\mu(f^{-n}\xi|\xi)(x)
=h_\mu(f, \xi)
\quad \mu\ae x.
\end{equation*}
\end{lemma}

\begin{proof}
Applying Lemma~\ref{Linduction}(i) with $\b=\c=\xi$,
and then Birkhoff ergodic theorem, we have that for almost every $x$,
\begin{equation*}
\lim_{n\to \infty}\frac{1}{n}I_\mu(f^{-n+1}\xi|\xi)(x)
=\lim_{n\to \infty}\frac{1}{n}\sum_{i=0}^{n-1} I_\mu(\xi|f\xi)(f^i(x))
=H_\mu(\xi|f\xi)=h(f, \xi),
\end{equation*}
where we used the fact $\xi_0^{i-1}=f^{-i+1}\xi$ and
$f\xi\vee f^i\xi=f\xi$ for all $i\ge 1$.

With $\b=\xi, \c=\eta$, we use Lemma~\ref{Linduction}(i) again to get
\begin{equation*}
\lim_{n\to \infty}\frac{1}{n}I_\mu(f^{-n+1}\xi|\eta)(x)
=\lim_{n\to \infty}\frac{1}{n}\Big[I_\mu(\xi|\eta)(x)+
  \sum_{i=1}^{n-1} I_\mu(\xi|f\xi\vee f^i\eta)(f^i(x))\Big].
\end{equation*}
By Lemma~\ref{Lconvergence}, for $\mu\ae x$,
there exist $N>0$ such that for any $i>N$,
$I_\mu(\xi|f\xi\vee f^i\eta)(f^i(x))=I_\mu(\xi|f\xi)(f^i(x))$.
Therefore the limit is also equal to $h(f, \xi)$.
\end{proof}

\begin{lemma}\label{Lmain}
Let $\a\in \P$, $\eta\in \P^u$ and $\xi\in \Q^u$.  Then
for $\mu$-a.e. $x$,
\[
\lim_{n\to\infty}\frac{1}{n}I_\mu(\a_0^{n-1}|\xi_0^{n-1}\vee \eta)(x)=0.
\]
\end{lemma}

\begin{proof}
By Lemma~\ref{Linduction}(ii) with $\c=\xi_0^{n-1}\vee \eta$,
\begin{equation}\label{fmain1}
\begin{split}
 &I_\mu(\alpha_0^{n-1}|\xi_0^{n-1}\vee \eta)(x) \\
=&I_\mu(\alpha|\xi\vee f^{n-1}\eta)(f^{n-1}(x))
+\sum_{i=0}^{n-2}I_\mu(\alpha|\alpha_1^{n-1-i}\vee \xi_1^{n-1-i}\vee f^i\eta)(f^i(x)),
\end{split}
\end{equation}
where we use the fact $f^i\xi_0^{n-1}=\xi_{-i}^{n-1-i}=\xi_1^{n-1-i}$
since $\xi$ is increasing.

Take $\phi_1=I_\mu(\alpha|\xi)(x)$, and
$\phi_n(x)=I_\mu(\alpha|\alpha_1^{n-1}\vee \xi_1^{n-1})(x)$.
Since $\diam (\alpha_1^{n-1}\vee \xi_1^{n-1})(x)\to 0$ as $n\to\infty$,
$\phi_n\to 0$ as $n\to \infty$ almost everywhere.

Also, by Lemma~\ref{Lphi1}, $\phi^*=\sup_{n} \phi_n\in L^1(\mu)$.
Hence we can apply Lemma~\ref{Lergodic} to get that $\mu$-a.e. $x$,
\[
\lim_{n\to\infty}\frac{1}{n}\Big[I_\mu(\alpha|\xi)(f^{n-1}(x))
+\sum_{i=0}^{n-2}I_\mu(\alpha|\alpha_1^{n-1-i}\vee \xi_1^{n-1-i})(f^i(x))\Big]
=0.
\]

By Lemma~\ref{Lconvergence}, for almost every $x$, there is $N>0$
such that for all $i>N$, $(f\xi\vee f^i\eta)(f^i(x))=(f\xi)(f^i(x))$.
The equation is still true if the partition $f\xi$ is replaced
by finer ones. So we have that
\[
I_\mu(\alpha|\alpha_1^{n-1-i}\vee \xi_1^{n-1-i})(f^i(x))
=I_\mu(\alpha|\alpha_1^{n-1-i}\vee \xi_1^{n-1-i}\vee f^i\eta)(f^i(x))
\]
for all large $i$.  Therefore, we can get
\[
\lim_{n\to\infty}\frac{1}{n}\Big[I_\mu(\alpha|\xi\vee f^{n-1}\eta)(f^{n-1}(x))
+\sum_{i=0}^{n-2}I_\mu(\alpha|\alpha_1^{n-1-i}\vee \xi_1^{n-1-i}\vee f^i\eta)(f^i(x))\Big]
=0.
\]
By \eqref{fmain1} we get the result of the lemma.
\end{proof}

\begin{proof}[Proof of Theorem B: Upper Limits]

By Lemma~\ref{Lcond1}
\[
I_\mu(\a_0^{n-1}|\eta)(x)
\le I_\mu(\a_0^{n-1}\vee \xi_0^{n-1}|\eta)(x)
= I_\mu(\xi_0^{n-1}|\eta)(x)+ I_\mu(\a_0^{n-1}|\xi_0^{n-1}\vee\eta)(x).
\]
Hence, by Lemma~\ref{Lmain}, Lemma~\ref{LSMB1} and Theorem~A,
\begin{equation*}
\begin{split}
\limsup_{n\to \infty}\frac{1}{n}I_\mu(\a_0^{n-1}|\eta)(x)
\le \limsup_{n\to \infty}\frac{1}{n}I_\mu(\xi_0^{n-1}|\eta)(x)
%= \limsup_{n\to \infty}\frac{1}{n}I_\mu(\xi_0^{n-1}|\eta)(x)
=h_\mu(f,\xi)=h_\mu(f,\a|\eta).
\end{split}
\end{equation*}
We get the same bound for the upper limit.
\end{proof}

\begin{proof}[Proof of Corollary B.1]

Theorem B implies that
$\lim_{n\to \infty}\frac{1}{n}I_\mu(\a_0^{n-1}|\eta)(x)$ exists.
So Lemma~\ref{LSMB1.3} gives
\begin{equation}\label{fB1.1}
\lim_{n\to \infty}\frac{1}{n}I_\mu(\a_0^{n-1}|\eta)(x)
= \lim_{n\to \infty}\frac{1}{n}I_\mu(\a_0^{n-1}|\xi)(x)
\qquad \mu\text{-a.e.}\ x.
\end{equation}
By Theorem~B we have
\begin{equation}\label{fB1.2}
h_\mu(f, \a|\eta)
=\int \lim_{n\to \infty}\frac{1}{n}I_\mu(\a_0^{n-1}|\eta)d\mu.
\end{equation}
Hence by Fatou's lemma, \eqref{fB1.1} and \eqref{fB1.2},
\begin{equation}\label{fB1.3}
\begin{split}
h_\mu(f, \a|\xi)
\ge &\liminf_{n\to \infty} \frac{1}{n}H_\mu(\a_0^{n-1}|\xi)
\ge \int \liminf_{n\to \infty} \frac{1}{n}I_\mu(\a_0^{n-1}|\xi)d\mu \\
=  &\int \liminf_{n\to \infty} \frac{1}{n}I_\mu(\a_0^{n-1}|\eta)d\mu
=h_\mu(f, \a|\eta).
\end{split}
\end{equation}
On the other hand, since
\[
H_\mu(\a_0^{n-1}|\xi)
\le H_\mu(\a_0^{n-1}|\eta)+H_\mu(\eta|\xi),
\]
and $H_\mu(\eta|\xi)< \infty$, we have
\[
h_\mu(f, \a|\xi)
= \limsup_{n\to \infty} \frac{1}{n}H_\mu(\a_0^{n-1}|\xi)
\le \lim_{n\to \infty} \frac{1}{n}H_\mu(\a_0^{n-1}|\eta)
=h_\mu(f, \a|\eta).
\]
Together with \eqref{fB1.3}, we obtain $h_\mu(f, \a|\xi)=h_\mu(f, \a|\eta)$.
Now the conclusion of the corollary follows from
Theorem~B and \eqref{fB1.1}.
\end{proof}

%\begin{comment}

%%%%%%%%%%%%%%%%%%%%%%%%%%%%%%%%%%%%%%%%%%%
%%%%%%%%%%%%%%%%%%%%%%%%%%%%%%%%%%%%%%%%%%%
\section{Unstable topological entropy}\label{SUTE}
\setcounter{equation}{0}
%%%%%%%%%%%%%%%%%%%%%%%%%%%%%%%%%%%%%%%%%%%

%%%%%%%%%%%%%%%%%%%%%%%%%%%%%%%%%%%%%%%%%%%
%%%%%%%%%%%%%%%%%%%%%%%%%%%%%%%%%%%%%%%%%%%
\subsection{Definition using spanning sets}
%%%%%%%%%%%%%%%%%%%%%%%%%%%%%%%%%%%%%%%%%%%

Recall that unstable topological entropy is defined in
Definition~\ref{Defutopent1} using $(n,\epsilon)$ u-separated sets.
We can also define unstable topological entropy by using
 $(n,\epsilon)$ u-spanning sets as follows.
A set $E \subset W^u(x)$ is called an \emph{$(n,\epsilon)$ u-spanning set}
of $\overline{W^u(x, \delta)}$ if
$\overline{W^u(x, \delta)} \subset \bigcup_{y\in E}B^u_{n}(y,\epsilon)$,
where $B^u_{n}(y,\epsilon)=\{z\in W^u(x): d^u_{n}(y,z) \leq\epsilon\}$
is the $(n,\epsilon)$ u-Bowen ball around $y$.
Let $S^u(f,\epsilon,n,x,\delta)$ be the cardinality of a minimal
$(n,\epsilon)$ u-spanning set of $\overline{W^u(x, \delta)}$.
It is standard to verify that
$$
N^u(f,2\epsilon,n,x,\delta) \leq S^u(f,\epsilon,n,x,\delta)
\leq N^u(f,\epsilon,n,x,\delta).
$$
So in Definition~\ref{Defutopent1} we can also use
$$
h^u_{\text{top}}(f, \overline{W^u(x,\delta)})
=\lim_{\epsilon \to 0}\limsup_{n\to \infty}\frac{1}{n}
\log S^u(f,\epsilon,n,x,\delta).
$$

The following lemma tells that in the definition we do not have to
let $\d\to 0$.

\begin{lemma}\label{smalldelta}
$h^u_{\text{top}}(f)=\sup_{x\in M}h^u_{\text{top}}(f, \overline{W^u(x,\delta)})$ for any $\delta >0$.
\end{lemma}
\begin{proof}
It is easy to see that $h^u_{\text{top}}(f)\leq \sup_{x\in M}h^u_{\text{top}}(f, \overline{W^u(x,\delta)})$ for any $\delta >0$ since $\delta \mapsto \sup_{x\in M}h^u_{\text{top}}(f, \overline{W^u(x,\delta)})$ is increasing.

Let us prove the other direction for some fixed $\delta >0$. For any $\rho >0$, let $y \in M$ be such that
\begin{equation}\label{e:anyradius1}
\sup_{x\in M}h^u_{\text{top}}(f, \overline{W^u(x,\delta)})\leq h^u_{\text{top}}(f, \overline{W^u(y,\delta)})+\frac{\rho}{3}.
\end{equation}
We can choose $\epsilon_0>0$ such that
\begin{equation}\label{e:anyradius2}
\begin{aligned}
h^u_{\text{top}}(f, \overline{W^u(y,\delta)})&=\lim_{\epsilon \to 0}\limsup_{n\to \infty}\frac{1}{n}\log S^u(f,\epsilon,n,y,\delta) \\ &\leq \limsup_{n\to \infty}\frac{1}{n}\log S^u(f,\epsilon_0,n,y,\delta)+\frac{\rho}{3}.
\end{aligned}
\end{equation}
Choose $\delta_1>0$ small enough such that $\delta_1<\delta$ and
\begin{equation}\label{e:anyradius3}
h^u_{\text{top}}(f)\geq\sup_{x\in M}h^u_{\text{top}}(f, \overline{W^u(x,\delta_1)})-\frac{\rho}{3}.
\end{equation}
Then there exist $y_j \in \overline{W^u(y,\delta)}, 1\leq j\leq N$ where $N$ only depends on $\delta$, $\delta_1$, and the Riemannian structure on $\overline{W^u(y,\delta)}$, such that
\begin{equation*}
\overline{W^u(y,\delta)} \subset \bigcup_{j=1}^N\overline{W^u(y_j, \delta_1)}.
\end{equation*}
It follows that
\begin{equation}\label{e:anyradius4}
\begin{aligned}
&\limsup_{n\to \infty}\frac{1}{n}\log S^u(f,\epsilon_0,n,y,\delta)
\leq \limsup_{n\to \infty}\frac{1}{n}
   \log   \left(\sum_{j=1}^N S^u(f,\epsilon_0,n,y_i,\delta_1)\right)\\
\leq &\limsup_{n\to \infty}\frac{1}{n}
   \log N S^u(f,\epsilon_0,n,y_i,\delta_1)
=\limsup_{n\to \infty}\frac{1}{n}
  \log S^u(f,\epsilon_0,n,y_i,\delta_1) \\
\leq &\lim_{\epsilon \to 0}\limsup_{n\to \infty}\frac{1}{n}
   \log S^u(f,\epsilon,n,y_i,\delta_1)
 = h^u_{\text{top}}(f, \overline{W^u(y_i,\delta_1)})\\
\end{aligned}
\end{equation}
for some  $1\leq i \leq N$. Combining \eqref{e:anyradius1}, \eqref{e:anyradius2}, \eqref{e:anyradius4} and \eqref{e:anyradius3},
\begin{equation*}
\begin{aligned}
&\sup_{x\in M}h^u_{\text{top}}(f,\overline{W^u(x,\delta)})
\leq h^u_{\text{top}}(f, \overline{W^u(y,\delta)})+\frac{\rho}{3} \\
\leq&\limsup_{n\to \infty}\frac{1}{n}\log S^u(f,\epsilon_0,n,y,\delta)
   +\frac{2\rho}{3}
 \leq h^u_{\text{top}}(f, \overline{W^u(y_i,\delta_1)})+\frac{2\rho}{3}\\
\leq &\sup_{x\in M}h^u_{\text{top}}(f, \overline{W^u(x,\delta_1)})+\frac{2\rho}{3}
 \leq h^u_{\text{top}}(f)+\rho.  %\ \ \ \ \ \text{by} \eqref{e:anyradius3}.
\end{aligned}
\end{equation*}
Since $\rho>0$ is arbitrary, one has
$\sup_{x\in M}h^u_{\text{top}}(f, \overline{W^u(x,\delta)})\leq h^u_{\text{top}}(f)$.
\end{proof}

\subsection{Definition using open covers}
We proceed to define the unstable topological entropy by using open covers. Let $\mathcal{C}_M^o$ denote the set of open covers of $M$. Given $\mathcal{U}\in \mathcal{C}_M^o$, denote $\mathcal{U}_m^n:=\bigvee_{i=m}^n f^{-i}\mathcal{U}$. For any $K \subset M$, set $N(\mathcal{U}|K):=\min\{\text{the cardinality of\ }\mathcal{V}: \mathcal{V}\subset \mathcal{U}, \bigcup_{V\in \mathcal{V}}\supset K\}$ and $H(\mathcal{U}|K):=\log N(\mathcal{U}|K)$.

\begin{definition}\label{Defutopent2}
We define
$$
\tilde{h}^u_{\text{top}}(f)=\lim_{\delta \to 0}
\sup_{x\in M}\tilde{h}^u_{\text{top}}(f, \overline{W^u(x,\delta)}),
$$
where
$$
\tilde{h}^u_{\text{top}}(f,  \overline{W^u(x,\delta)})
=\sup_{\mathcal{U}\in \mathcal{C}_M^o}\limsup_{n\to \infty}
\frac{1}{n}H(\mathcal{U}_{0}^{n-1}|\overline{W^u(x,\delta)}).$$
\end{definition}

\begin{remark}
It is easy to see that $H(\mathcal{U}|\overline{W^u(x,\delta)})=H(f^{-1}\mathcal{U}|f^{-1}(\overline{W^u(x,\delta)}))$.
But we don't know whether the sequence $H(\mathcal{U}_{0}^{n-1}|\overline{W^u(x,\delta)})$ is subadditive or not, and so we use $\limsup$ in the definition above. That is the main difference from the case for classical topological entropy.
\end{remark}
Now we verify that the two definitions in Definition \ref{Defutopent1} and \ref{Defutopent2} for unstable topological entropy coincide.

\begin{lemma}\label{cover}
Let $\delta >0$ be small enough. Then there exists a constant $C>1$ such that for any $\epsilon>0$ small enough, any $\mathcal{U}_\epsilon\in \mathcal{C}_M^o$ with Lebesgue number $2\epsilon$, and any $\mathcal{V}_\epsilon\in \mathcal{C}_M^o$ with $\text{diam}(\mathcal{V}_\epsilon) \leq \dfrac{\epsilon}{C}$,
$$N((\mathcal{U}_\epsilon)_0^{n-1}|\overline{W^u(x,\delta)})\leq S^u(f,\epsilon,n,x,\delta) \leq N^u(f,\epsilon,n,x,\delta) \leq N((\mathcal{V}_\epsilon)_0^{n-1}|\overline{W^u(x,\delta)}).$$
\end{lemma}
\begin{proof}
Observe that for $\delta >0$ small enough, there exists $C>1$ such that for any $x\in M$,
$$d(y,z) \leq d^u(y,z)\leq Cd(y,z) \text{\ \ for any }y,z \in \overline{W^u(x,\delta)}.$$
Then the lemma follows by a similar argument as in the proof of Theorem 7.7 in \cite{W}.
\end{proof}

\begin{corollary}
$\tilde{h}^u_{\text{top}}(f, \overline{W^u(x,\delta)})
=h^u_{\text{top}}(f, \overline{W^u(x,\delta)})$.
As a consequence,
$$\tilde{h}^u_{\text{top}}(f)=h^u_{\text{top}}(f).$$
\end{corollary}

\subsection{Proof of Theorem C: relation to unstable volume growth}
In this subsection, we prove Theorem C, which states that
the unstable topological entropy actually coincides with the unstable
volume growth defined in \cite{HSX}.
The notation $\chi_u(f)$ for unstable volume growth is used in \cite{HSX}.

\proof [Proof of Theorem C]
Choose a small $\delta >0$. By the definition of $\chi_u(f)$,
for any given $\rho >0$, there exists a point $x$ such
that
\[
\chi_u(x, \delta)\geq \chi_u(f)-\rho.
\]
For $\ep >0$, let $E$ be a minimal $(n,\ep)$ u-spanning set of
$\overline{W^u(x,\delta)}$, then $f^n(E)$ is an $\ep$-spanning set of
$f^n(\overline{W^u(x,\delta)})$. Thus $f^n(W^u(x,\delta))\subset \bigcup_{y\in
f^n(E)}W^u(y,\ep)$. The volume of any $\ep$ u-ball is bounded from above by $ c_1
\ep^k$, where $c_1>0$ is a constant depending on the Riemannian
metric and $k$ is the dimension of the unstable manifolds. Then the
total volume covered by the $\ep$ u-balls around the points in
$f^n(E)$ is less than $c_1\ep^kS^u(f,\epsilon,n,x,\delta)$.
Therefore
$$\text{Vol}(f^n(W^u(x,\delta)))\leq c_1
\ep^kS^u(f,\epsilon,n,x,\delta).$$
We get
\begin{equation*}
\begin{aligned}
 h_{\text{top}}^u(f,\overline{W^u(x,\delta)})
 &=\lim_{\ep\to 0} \limsup_{n \to \infty} \frac{1}{n} \log
S^u(f,\ep,n,x,\delta) \\
&\geq \lim_{\ep\to 0} \limsup_{n \to \infty} \frac{1}{n}
\log(\text{Vol}(f^n(W^u(x,\delta)))/(c_1
\ep^k))\\
&=\chi_u(x, \delta)\geq \chi_u(f) - \rho.
\end{aligned}
\end{equation*}
Since $\rho>0$ is arbitrary, $$h_\text{top}^u(f) =\sup_{x\in M}
h_{\text{top}}^u(f,\overline{W^u(x,\delta)})\geq \chi_u(f).$$

On the other hand, for any given $\rho>0$, by the definition of
$h_{\text{top}}^u(f)$, there exist a point $x$ and $0<\ep_0<\delta$ such that
\[
\limsup_{n\to \infty}\frac{1}{n}N^u(f,\ep_0,n,x,\delta)>h_\text{top}^u(f)-\rho.
\]
 Let $F\subset \overline{W^u(x,\delta)}$ be an $(n,\ep_0)$ separated
set, then $f^n(F)$ is $\ep_0$ separated. The volume of any $\ep_0/2$ u-ball can be bounded from below by $c_2\ep_0^k$,
where $c_2>0$ is a constant depending on the Riemannian metric.
Since the $\ep_0/2$ u-balls around the points in $f^n(F)$ are
disjoint subsets of $f^n(W^u(x,\delta+\ep_0))$, we get
$$\text{Vol}(f^n(W^u(x,\delta+\ep_0)))\geq c_2\ep_0^kN^u(f,\ep_0,n,x,\delta).$$
Therefore,
\begin{equation*}
 \begin{split}
\chi_u(f)
\geq \chi_u(x,\delta+\ep_0)
= &\limsup_{n \to\infty} \frac{1}{n}\log\text{Vol}(f^n(W^u(x,\delta+\ep_0)))\\
\geq &\limsup_{n \to \infty} \frac{1}{n}
   \log\left( c_2 \ep_0^kN^u(f,\ep_0,n,x,\delta) \right)
  >h_{\text{top}}^u(f) - \rho.
\end{split}
\end{equation*}
Since $\rho>0$ is arbitrary, $$\chi_u(f)\geq h_{\text{top}}^u(f).$$
This completes the proof.\qed
\medskip

\begin{proof}[Proof of Corollary C.1]
The inequality $h_\top^u(f)\le h_\top(f)$ follows from the definition directly.

If there is no positive Lyapunov exponents in the center direction, then by
\eqref{fineqHSX} and Theorem C, we also have
$h_\top(f)\le \chi^u(f) = h_\top^u(f)$.
\end{proof}

\begin{proof}[Proof of Corollary C.2]
Clearly for any invariant measure $\mu$,
$\sum_{\lambda_i^c>0}\lambda_i^c \le \sigma$.
By Theorem C, \eqref{fineqHSX} implies
\[
h_\mu(f)\leq h_\top^u(f)+\sigma.
\]
Then we use the variational principle for entropy.
\end{proof}

\begin{proof}[Proof of Corollary C.3]
For any $x\in M$, $\d>0$, denote
\begin{equation*}
h_{\text{top}}(f, \overline{B(x,\delta)})
=\lim_{\epsilon \to 0}\limsup_{n\to \infty}\frac{1}{n}\log N(f,\epsilon,n,x,\delta),
\end{equation*}
where $N(f,\epsilon,n,x,\delta)$ denote the maximal number of points
in $\overline{B(x,\delta)}$ with pairwise $d_{n}$-distances
at least $\epsilon$.  By using a finite cover argument
we know that for any $\d>0$, there is an $x\in M$ such that
\[
h_{\text{top}}(f, \overline{B(x,\delta)})
=\lim_{\epsilon \to 0}\limsup_{n\to \infty}\frac{1}{n}\log N(f,\epsilon,n,M),
\end{equation*}
where $N(f,\epsilon,n,M)$ denote the maximal number of points
in $M$ with pairwise $d_{n}$-distances at least $\epsilon$.
Hence we can get that
\[
h_\top(f)
=\lim_{\delta \to 0}\sup_{x\in M}h_{\text{top}}(f, \overline{B(x,\delta)}).
\]

By the definition of unstable and transversal topological entropies,
Definition~\ref{Defutopent1} and Definition~\ref{Deftransent}, we have
\[
h_{\text{top}}(f, \overline{B(x,\delta)})
\le h_{\text{top}}^u(f, \overline{B(x,\delta)})
+h_{\text{top}}^t(f, \overline{B(x,\delta)}).
\]
Then taking supremum over $x\in M$ and letting $\delta$ go to $0$,
we get the inequality.
\end{proof}

%%%%%%%%%%%%%%%%%%%%%%%%%%%%%%%%%%%%%%%%%%%%%%%%%%
%%%%%%%%%%%%%%%%%%%%%%%%%%%%%%%%%%%%%%%%%%%%%%%%%%
\section{The variational principle}\label{SVP}
\setcounter{equation}{0}
%%%%%%%%%%%%%%%%%%%%%%%%%%%%%%%%%%%%%%%%%%%%%%%%%%

In this section, we prove Theorem~D, the variational principle
for unstable entropies $h^u_{\text{top}}(f)$ and $h_{\mu}^u(f)$.

%\subsection{Proof of the variational principle}

At first, we prove one direction of the variational principle as follows.

\begin{proposition}\label{variationalprinciple1}
Let $\mu$ be any $f$-invariant probability measure. Then
$$h_{\mu}^u(f)\leq h^u_{\text{top}}(f).$$
\end{proposition}

\begin{proof}
Let $\mu=\int_{\mathcal{M}^e_f(M)}\nu d\tau(\nu)$ be the unique ergodic decomposition
where $\tau$ is a probability measure on the Borel subsets of $\mathcal{M}_f(M)$ and $\tau(\mathcal{M}^e_f(M))=1$.
Since $\mu \mapsto h_\mu^u(f)$ is affine and upper semi-continuous by Propositions \ref{Paffine} and \ref{Pusc}, then
\begin{equation}\label{e:ergodiccomponent}
h_\mu^u(f)=\int_{\mathcal{M}^e_f(M)}h_\nu^u(f) d\tau(\nu)
\end{equation}
by a classical result in convex analysis (cf. Fact A.2.10 on p. 356 in \cite{Do}). Therefore, we only need to prove the proposition for
ergodic measures.

Suppose $\mu$ is ergodic. Let $\rho>0$. Take $\eta\in \P^u$ subordinate to unstable manifolds, and take $\e>0$.
By Corollary~\ref{Cpointwise}, we have
\begin{equation*}
\lim_{n \to \infty}-\frac{1}{n}\log\mu_y^\eta(B_{n}^u(y,\e))
\geq h_\mu^u(f|\eta)  \qquad \mu-\text{a.e.}\ y.
\end{equation*}
Hence for $\mu$-a.e. $y$, there exists $N(y)=N(y,\e)>0$ such that if
$n\geq N(y)$, then
\begin{equation*}
\mu_y^\eta(B_{n}^u(y,\e)) \leq e^{-n(h_\mu^u(f|\eta)-\rho)}.
\end{equation*}
Denote $E_n=E_n(\e)=\{y\in M: N(y)=N(y,\e)\leq n\}$.
Then $\mu\big(\cup_{n=1}^\infty E_n\big)=1$ by the corollary.
So there exists $n>0$ large enough such that $\mu(E_n)>1-\rho$.
Hence, there exists $x\in M$ such that
$\mu_x^\eta(E_n)=\mu_x^\eta(E_n\cap \eta(x))>1-\rho$.
Fix such $n$ and $x$.
Note that if $y \in \eta(x)$, then $\mu_y^\eta=\mu_x^\eta$.
We have
\begin{equation*}
\mu_x^\eta(B_{n}^u(y,\e)) \leq e^{-n(h_\mu(f|\eta)-\rho)} \qquad
\forall y\in E_n\cap \eta(x).
\end{equation*}

Let $S^u(f,\e,n,\eta(x))$ be the cardinality of a minimal
$(n,\e)$ u-spanning set of $\overline{\eta(x)}$.
Then there exists a set $S\subset \eta(x)$ with cardinality no more than
$S^u(f,\e/2,n,\eta(x))$ such that
\begin{equation*}
\begin{aligned}
\eta(x)\cap E_n \subset &\bigcup_{z\in S}B_{n}^u(z,\e/2).
\end{aligned}
\end{equation*}
and $B_{n}^u(z,\e/2)\cap E_n \neq \emptyset$.
Let $y(z)$ be an arbitrary point in $B_{n}^u(z,\e/2)\cap E_n$. We have
\begin{equation*}
\begin{split}
1-\rho&<\mu_x^\eta(\eta(x)\cap E_n)
\leq \mu_x^\eta(\bigcup_{z\in S}B_{n}^u(z,\e/2))
\leq \sum_{z\in S}\mu_x^\eta(B_{n}^u(z,\e/2))\\
&\leq \sum_{z\in S}\mu_x^\eta(B_{n}^u(y(z),\e))
\leq S^u(f,\e/2,n,\eta(x))e^{-n(h_\mu(f|\eta)-\rho)}.
\end{split}
\end{equation*}
Hence $S^u(f,\e/2,n,\eta(x)) \geq (1-\rho)e^{n(h_\mu(f|\eta)-\rho)}$.

Now we take $\delta>0$ such that with $W^u(x, \delta)\supset \eta(x)$.
Recall that $S^u(f,\e,n,x,\delta)$ denotes the cardinality of a minimal
$(n,\e)$ u-spanning set of $\overline{W^u(x, \delta)}$.
Clearly  $S^u(f,\e,n,x,\delta)\ge S^u(f,\e,n, \eta(x))$.
So we get
$S^u(f,\e/2,n,x,\delta) \geq (1-\rho)e^{n(h_\mu(f|\eta)-\rho)}$.
It follows that
$$h^u_{\text{top}}(f,\overline{W^u(x,\delta)}) \geq h_\mu(f|\eta)-\rho.$$
Then by definition,
$$h^u_{\text{top}}(f) \geq h^u_{\text{top}}(f,\overline{W^u(x,\delta)})
\geq h_\mu(f|\eta)-\rho
= h^u_\mu(f)-\rho.$$
Since $\rho$ is arbitrary, we have
\[
h^u_{\text{top}}(f) \geq h_\mu^u(f).  \qedhere
\]
\end{proof}

Next we use the ideas in Misiurewicz's proof of the classical variational
 principle (\cite{Mis2}) to prove Theorem~D.

\begin{proof}[Proof of Theorem~D]
By Proposition~\ref{variationalprinciple1}, it is enough to prove
that for any $\rho>0$, there exists $\mu \in \mathcal{M}_f(M)$
such that $h_{\mu}^u(f)\geq h^u_{\text{top}}(f)-\rho$.

For some $\delta>0$ small enough, we can find a point $x\in M$ such that
$$h^u_{\text{top}}(f, \overline{W^u(x,\delta)})\geq h^u_{\text{top}}(f)-\rho.$$
Take $\e>0$ small enough.
Let $S_n$ be an $(n,\e)$ u-separated set of $\overline{W^u(x,\delta)}$
with cardinality $N^u(f,\e,n,x,\delta)$.
Define
$$\nu_n:=\frac{1}{N^u(f,\e,n,x,\delta)}\sum_{y\in S_n}\delta_y,$$
and
$$\mu_n:=\frac{1}{n}\sum_{i=0}^{n-1}f^i\nu_n.$$
Since the set $\mathcal{M}(M)$ of all probability measures of $M$
is a compact space with weak* topology, there exists a subsequence $\{n_k\}$
of natural numbers such that $\lim_{k\to \infty}\mu_{n_k}=\mu$.
Obviously $\mu \in \mathcal{M}_f(M)$.

As $\delta$ is very small, we can choose a partition $\eta\in\P^u$ such that $W^u(x,\delta)\subset \eta(x)$.
That is, $W^u(x,\delta)$ is contained in a single element of $\eta$.
Then choose $\a\in \P$ such that $\mu(\partial\alpha)=0$,
and $\text{diam}(\alpha) < \frac{\e}{C}$
where $C>1$ is as in the proof of Lemma~\ref{cover}.
Hence we have $\log N^u(f,\e,n,x,\delta)=H_{\nu_n}(\alpha_0^{n-1}|\eta)$.

Fix a natural numbers $q>1$.  For any natural number $n>q$,
$j=0, 1, \cdots, q-1$, put $a(j)=[\frac{n-j}{q}]$, where
$[a]$ denotes the integer part of $a>0$.
Then
$$\bigvee_{i=0}^{n-1}f^{-i}\alpha
=\bigvee_{r=0}^{a(j)-1}f^{-(rq+j)}\alpha_0^{q-1} \vee \bigvee_{t\in S_j}f^{-t}\alpha,$$
where $S_j=\{0,1,\cdots, j-1\}\cup \{j+qa(j), \cdots, n-1\}$.

For a partition $\a\in \P$, denote by $\a^u$ the partition in $\P^u$
whose elements are given by $\a^u(x)=\a(x)\cap W^u_\loc(x)$. Note that
\[
f^{rq}\big(\bigvee_{i=0}^{r-1}f^{-iq}\alpha_0^{q-1}\vee f^j\eta\big)
=f^{rq}\big(\alpha_0^{rq-1}\vee f^j\eta\big)
=f\a\vee \cdots \vee f^{rq}\a\vee f^{rq+j}\eta
\ge f\a^u.
\]
Note also that the same arguments as for Lemma~\ref{Linduction}(i)
can be applied for any probability measure $\nu$ that is not necessary
invariant. We can get that
\begin{equation}\label{fVP1}
\begin{aligned}
&H_{\nu}(\bigvee_{r=0}^{a(j)-1}f^{-rq}\alpha_0^{q-1}|f^j\eta)\\
= &H_{\nu}(\alpha_0^{q-1}|f^j\eta)
+ \sum_{r=1}^{a(j)-1}H_{f^{rq}\nu}\Big(\alpha_0^{q-1}
 \Big|f^{rq}\big(\bigvee_{i=0}^{r-1}f^{-iq}\alpha_0^{q-1}\vee f^j\eta\big) \Big)\\
\le&  H_{\nu}(\alpha_0^{q-1}|f^j\eta)
+ \sum_{r=1}^{a(j)-1}H_{f^{rq}\nu}(\alpha_0^{q-1}|f\a^u).
\end{aligned}
\end{equation}
Also,
\begin{equation}\label{fVP2}
H_{\nu}(\bigvee_{r=0}^{a(j)-1}f^{-(rq+j)}\alpha_0^{q-1}|\eta)
=H_{f^j\nu}(\bigvee_{r=0}^{a(j)-1}f^{-rq}\alpha_0^{q-1}|f^j\eta).
\end{equation}
Replacing $\nu$ by $\nu_n$ and $f^j\nu_n$ in \eqref{fVP2} and  \eqref{fVP1}
respectively we get
\begin{equation*}
\begin{aligned}
&\log N^u(f,\e,n,x,\delta)=H_{\nu_n}(\alpha_0^{n-1}|\eta) %\\
=H_{\nu_n}\big(\bigvee_{r=0}^{a(j)-1}f^{-(rq+j)}\alpha_0^{q-1} \vee
   \bigvee_{t\in S_j}f^{-t}\alpha|\eta\big)\\
&\leq \sum_{t\in S_j}H_{\nu_n}(f^{-t}\alpha|\eta)
 +H_{\nu_n}\big(\bigvee_{r=0}^{a(j)-1}f^{-rq-j}\alpha_0^{q-1}|\eta\big)\\
&\leq \sum_{t\in S_j}H_{\nu_n}(f^{-t}\alpha|\eta)
 +H_{f^j\nu_n}\big(\bigvee_{r=0}^{a(j)-1}f^{-rq}\alpha_0^{q-1}|f^j\eta\big)\\
&\leq \sum_{t\in S_j}H_{\nu_n}(f^{-t}\alpha|\eta)
+ H_{f^j\nu_n}\big(\alpha_0^{q-1}|f^j\eta\big)
+ \sum_{r=1}^{a(j)-1}H_{f^{rq+j}\nu_n}(\alpha_0^{q-1}|f\a^u).
\end{aligned}
\end{equation*}
It is clear that $\text{card} S _j \leq 2q$.
Denote by $d$ the number of elements of $\a$.
Summing the inequalities over $j$ form $0$ to $q-1$ and
dividing by $n$, by Proposition~\ref{Paffine} we get
\begin{equation}\label{e:mu}
\begin{aligned}
&\frac{q}{n}\log N^u(f,\e,n,x,\delta) \\
\leq &\frac{1}{n}\sum_{j=0}^{q-1} \sum_{t\in S_j}H_{\nu_n}(f^{-t}\alpha|\eta)
   +\frac{1}{n}\sum_{j=0}^{q-1}H_{f^j\nu_n}(\alpha_0^{q-1}|f^j\eta)
   +\frac{1}{n}\sum_{i=0}^{n-1}H_{f^i\nu_n}(\alpha_0^{q-1}|f\a^u)\\
\leq &\frac{2q^2}{n}\log d
+\frac{1}{n}\sum_{j=0}^{q-1}H_{f^j\nu_n}(\alpha_0^{q-1}|f^j\eta)
 +H_{\mu_n}(\alpha_0^{q-1}|f\a^u).
\end{aligned}
\end{equation}
Let $\{n_k\}$ be a sequence of natural numbers such that
\begin{enumerate}
  \item $\mu_{n_k} \to \mu$ as $k\to \infty$;
  \item $\disp \lim_{k\to \infty}\frac{1}{n_k}\log N^u(f,\e,n_k,x,\delta)
     = \limsup_{n\to \infty}\frac{1}{n}\log N^u(f,\e,n,x,\delta)$.
%\item $\nu_{n_k} \to \nu$ as $k\to \infty$ for some $\nu\in \mathcal{M}(M)$.
\end{enumerate}
Since $\mu(\partial \alpha)=0$, and $\mu$ is invariant,
$\mu(\partial \alpha_0^{q-1})=0$ for any $q\in {\mathbb N}$.
Hence by Proposition~\ref{Pusc},
$$
\limsup_{k \to \infty}H_{\mu_{n_k}}(\alpha_0^{q-1}|f\a^u)
\leq H_\mu(\alpha_0^{q-1}|f\a^u).
$$
Thus replacing $n$ by $n_k$ in \eqref{e:mu} and letting $k\to \infty$,
we get
$$
qh^u_{\text{top}}(f, \overline{W^u(x,\delta)})
\leq H_{\mu}(\alpha_0^{q-1}|f\a^u).$$
By Corollary~A.2,
\begin{equation*}
h^u_{\text{top}}(f, \overline{W^u(x,\delta)})
\leq \lim_{q \to \infty}\frac{1}{q}H_{\mu}(\alpha_0^{q-1}|f\a^u)
=h_\mu(f, \a|f\a^u).
\end{equation*}
We may choose $\alpha\in \mathcal{P}$ such that $f\alpha^u\in \mathcal{P}^u$.
By Theorem~A, $h_\mu(f, \alpha|f\alpha^u)=h_\mu(f|f\alpha^u)=h_\mu^u(f)$.
Thus $h_\mu^u(f)\geq h^u_{\top}(f)-\rho$.
Since $\rho$ is arbitrary, we get the first equation of the theorem.

We prove the second equation in the theorem.

Let $\rho>0$ be sufficiently small. Then there exists an invariant measure $\mu$ such that
$h_\mu^u(f) > h^u_{\text{top}}(f)-\rho/2$. By \eqref{e:ergodiccomponent}, there exists an ergodic measure $\nu$ such that
$$h_{\nu}^u(f)>h_\mu^u(f)-\rho/2> h^u_{\text{top}}(f)-\rho.$$
Since $\rho$ is arbitrary, we have
$\disp h^u_{\text{top}}(f)=\sup\{h_{\nu}^u(f): \nu \in \mathcal{M}^e_f(M)\}$.
\end{proof}

\ \
\\[-2mm]
\textbf{Acknowledgement.} The authors would like to thank the referee for valuable suggestions. Y. Hua is supported by NSFC No. 11401133 and W. Wu is supported by NSFC No. 11701559.

\end{document}